\documentclass[12pt]{amsart}

\usepackage{amsmath}
\usepackage{amssymb}
\usepackage{amsfonts}
\usepackage{amsthm}
\usepackage{enumerate}
\usepackage{color}
\usepackage{psfrag}
\usepackage{mathrsfs}
\usepackage{amscd}
\usepackage[hyperindex,backref=page]{hyperref}
\usepackage{graphicx,color}
\usepackage{tikz} 
\usepackage{enumitem}
\usepackage[utf8]{inputenc}

\newtheorem{thm}{Theorem}[section]
\newtheorem{cor}[thm]{Corollary}
\newtheorem{lemma}[thm]{Lemma}
\newtheorem{prop}[thm]{Proposition}

\theoremstyle{definition}

\newtheorem{ex}[thm]{Example}

\newtheorem*{theorem-reg-equality}{Theorem \ref{reg-equality}} 
\newtheorem*{theorem-S_2main}{Theorem \ref{S_2main}} 
\newtheorem*{theorem-ineq of dual mod}{Theorem \ref{ineq of dual mod}} 
\newtheorem*{corollary-linear resol}{Corollary \ref{linear resol}} 
\newtheorem*{corollary-almost non-increasing}{Corollary \ref{almost non-increasing}}  
\newtheorem*{thm-second-power}{Theorem \ref{second-power}}  
\newtheorem*{cor-Serre depth Decrease}{Corollary \ref{Serre depth Decrease}}
\newtheorem*{thm-regQuotient}{Theorem \ref{regQuotient}} 
\newtheorem*{second-third-fourth}{Theorems \ref{PowerTwo},\ref{3rd symbolic power},\ref{4th symbolic power}}
\newtheorem*{thm-char-of-CM-S_2}{Theorem \ref{char of CM and S_2}} 

\theoremstyle{remark}

\DeclareMathOperator{\link}{link}
\DeclareMathOperator{\pd}{proj.dim}
\DeclareMathOperator{\set}{set}

\DeclareMathOperator{\depth}{depth}

\DeclareMathOperator{\supp}{supp}
\DeclareMathOperator{\reg}{reg}

\DeclareMathOperator{\indeg}{indeg}
\DeclareMathOperator{\height}{ht}

\DeclareMathOperator{\Min}{Min}
\DeclareMathOperator{\bight}{bight}
\DeclareMathOperator{\Tor}{Tor}
\DeclareMathOperator{\Hom}{Hom}
\DeclareMathOperator{\Ext}{Ext}
\DeclareMathOperator{\st}{star}

\begin{document}


\title[Regularity and depth of symbolic powers of complementary edge ideals]{Regularity, depth and Serre depth of symbolic powers of complementary edge ideals }

\author{Antonino Ficarra, Somayeh Moradi, and Yuji Muta}

\address[A. Ficarra]{BCAM – Basque Center for Applied Mathematics, Mazarredo 14, 48009 Bilbao, Basque Country – Spain, Ikerbasque, Basque Foundation for Science, Plaza Euskadi 5, 48009 Bilbao, Basque Country – Spain}
\email{aficarra@bcamath.org, antficarra@unime.it}

\address[S. Moradi]{ Department of Mathematics, Faculty of Science, Ilam University, P.O.Box 69315-516, Ilam, Iran}
\email{so.moradi@ilam.ac.ir}

\address[Y. Muta]{Department of Mathematics, Okayama University, 3-1-1 Tsushima-naka, Kita-ku, Okayama 700-8530 Japan.}
\email{p8w80ole@s.okayama-u.ac.jp}

\begin{abstract}
Let $G$ be a finite simple graph and let $I_c(G)$ denote its complementary edge ideal. In this paper, we study the Castelnuovo--Mumford regularity, the depth and Serre depth for $(S_{2})$ of the symbolic powers of $I_c(G)$. We derive explicit formulas for $\reg I_c(G)^{(k)}$ and $\reg I_c(G)^{(k)}/I_c(G)^k$ for every $k\geq1$, in terms of the number and the sizes of the connected components of $G$. As applications, we determine their asymptotic behavior, and completely characterize the symbolic powers of complementary edge ideals having linear resolutions. We then investigate the symbolic depth function $k\mapsto\depth S/I_c(G)^{(k)}$. We determine explicitly its values for $k=2,3,4$, and establish bounds for the asymptotic depth.
More generally, for an arbitrary squarefree monomial ideal $I$, we compare the local cohomology modules of $S/I^{(k)}$ for different symbolic powers, and derive corresponding inequalities for the depth and the Serre depth. Applying this to complementary edge ideals, we obtain $\depth S/I_{c}(G)^{(t)}\geq\depth S/I_{c}(G)^{(k)}$ for $k\geq 2t-1$, despite the fact that the symbolic depth function need not be non-increasing. Finally, we characterize complementary edge ideals satisfying Serre's condition $(S_r)$, and determine the Serre depth for $(S_{2})$ of all symbolic powers of complementary edge ideals. In particular, we obtain a combinatorial characterizations of Serre's condition $(S_{2})$ and of the Cohen--Macaulayness of $S/I_{c}(G)^{(k)}$ for every $k\geq 1$.
\end{abstract}






\thanks{}
	

\maketitle


\section{Introduction}

Squarefree monomial ideals form a natural bridge between commutative algebra and combinatorics. Through the Stanley--Reisner correspondence, algebraic invariants of such an ideal can be read off from the combinatorics of an associated simplicial complex, and vice versa. When the ideal corresponds to a graph, this dictionary becomes especially transparent. 
Classical examples include edge ideals, cover ideals, and their generalizations, whose algebraic invariants often encode subtle combinatorial properties of the underlying graph. A recent development in this direction is the study of squarefree monomial ideals generated in degree $n-2$ in $S=K[x_1,\ldots,x_n]$ via the notion of complementary edge ideals, introduced independently in \cite{hqm2026} and \cite{fm2025}. Every squarefree monomial ideal generated in degree $n-2$ is a complementary edge ideal. Since their introduction, several aspects of these ideals have been investigated, including their homological properties, ordinary and symbolic Rees algebras, associated primes, and homological shifts; see, for instance, \cite{bcj2026,f2026,fm2025,fmRees2025,fmm2026,hqm2026,lwz2025,l2026,rs2025}.

Let $G$ be a finite simple graph on the vertex set $[n]=\{1,\ldots,n\}$ with $E(G)\neq \emptyset$, and let $S=K[x_1,\ldots,x_n]$ be the polynomial ring over a field $K$. The \emph{complementary edge ideal} of $G$ is the squarefree monomial ideal
$$I_c(G)=\left(\frac{x_{1}\cdots x_{n}}{x_{i}x_{j}}:\{i,j\}\in E(G)\right).$$
Thus, the generators of $I_c(G)$ correspond to the complements of the edges of $G$. Despite the simplicity of this definition, the algebraic properties of $I_c(G)$ exhibit a rich dependence on the combinatorial structure of $G$. Complementary edge ideals exhibit, in several respects, a more tractable algebraic structure than their classical counterparts, namely edge ideals. A particularly significant instance was established in \cite{fm2025}, where the Castelnuovo--Mumford regularity of every ordinary power of $I_c(G)$ was explicitly determined by graph invariants, and the depth function $k\mapsto \depth S/I_c(G)^k$ was shown to be non-increasing. This behavior is especially noteworthy in comparison with edge ideals, for which the analogous monotonicity for depth is expected, but remains open in general, known only for certain classes of graphs.
The purpose of this paper is to investigate the Castelnuovo--Mumford regularity, depth and Serre depth for $(S_2)$ of the symbolic powers $I_c(G)^{(k)}$ for all $k\geq 1$.
 
The behavior of homological invariants under ordinary and symbolic powers is one of the central themes in the theory of monomial ideals. Even for squarefree monomial ideals $I$ arising naturally from graphs, determining the functions
\[
k\longmapsto \reg I^{(k)}
\qquad\text{and}\qquad
k\longmapsto \depth S/I^{(k)}
\]
can be a difficult problem. Complementary edge ideals provide an interesting setting in which the regularity functions $\reg I^{(k)}$ and $\reg I^{(k)}/I^k$ can be determined explicitly, while the depth function displays a considerably more intricate behavior.

Our results on regularity provide explicit combinatorial descriptions for the regularity of all symbolic powers $I^{(k)}$ and of all quotient modules $I^{(k)}/I^k$ for a complementary edge ideal $I=I_c(G)$, yielding a characterization of graphs $G$ for which the symbolic power $I^{(k)}$ has a linear resolution.  Our approach provides a multigraded combinatorial framework for the local cohomology of  $S/I_c(G)^{(k)}$.
More precisely, for each multidegree ${\bf a}\in\mathbb Z^n$ we associate to $G$ a hypergraph whose edges are determined by the nonedges and triangles of $G$ satisfying suitable inequalities depending on ${\bf a}$ and $k$. By combining Takayama's formula \cite{t2005} with combinatorial Alexander duality \cite{bt2009} and Hochster's formula \cite{H1977}, we express the multigraded components of the local cohomology modules of $S/I_c(G)^{(k)}$ in terms of the reduced homology, or equivalently suitable multigraded Betti numbers of this auxiliary hypergraph, yielding the regularity description. We note that the formula for $\reg I^{(k)}$ was also obtained independently, and almost simultaneously by Hien and Kumar in~\cite{hk2026}, with a different approach using symbolic polyhedra (to determine the slope of the regularity function), degree complexes, Takayama's formula, and monomial localization.

Let $c(G)$ denote the number of connected components of $G$ which are not isolated vertices, and set $\operatorname{sc}(G)=\min\{|V(C)|: C \text{ is a connected component of $G$ with at least two vertices}\}$. 
We prove
\begin{theorem-reg-equality} 
Let $G$ be a graph on the vertex set $[n]$. For each $k\geq1$, we have 
\[
\reg I_{c}(G)^{(k)}=
\begin{cases}
(n-2)k, & \text{if } c(G)=1, \\
\max\{(n-2)k, (n-sc(G))k+sc(G)-1\}, & \text{ if } c(G)>1.
\end{cases}
\]
\end{theorem-reg-equality}

Together with the formula for the regularity of ordinary powers obtained in \cite{fm2025}, this completely describes the relation between $\reg I_c(G)^{(k)}$ and $\reg I_c(G)^k$. In Corollary~\ref{reg-compare} we show that $\reg I_c(G)^{(k)}\leq \reg I_c(G)^k$ for every $k\geq1,$
characterize the graphs for which  equality occurs, prove that the sequence $(\reg I_c(G)^{(k)}/k)_{k\ge1}$ is non-increasing, and obtain
\[
\lim_{k\to\infty}
\frac{\reg I_c(G)^{(k)}}{k}
=
n-2.
\]
As a further consequence, in Corollary~\ref{linear resol}, we completely classify the symbolic powers of complementary edge ideals having a linear resolution.

A further highlight of Section \ref{sec3} is Theorem~\ref{regQuotient}, which gives an explicit formula for the regularity of the quotient $I_{c}(G)^{(k)}/I_{c}(G)^{k}$, a module that measures the gap between symbolic and ordinary powers.

\begin{thm-regQuotient}
Let $G$ be a graph on $[n]$ and let $k$ be a positive integer with $I_{c}(G)^{(k)}/I_{c}(G)^{k}\neq0$. Then we have $$\reg I_{c}(G)^{(k)}/I_{c}(G)^{k}=\reg I_{c}(G)^{k}-1= (n-2)k+\min\{k,c(G)-1\}-1.$$
\end{thm-regQuotient}

The second part of the paper is devoted to the depth of the symbolic powers of complementary edge ideals.
We already know from \cite[Proposition 5.5]{fmm2026} that the symbolic depth function of complementary edge ideals is not necessarily non-increasing, where an example of a complete multipartite graph $G$ was given for which $\depth S/I_{c}(G)^{(3)}<\depth S/I_{c}(G)^{(4)}$. Motivated by this, we study the behavior of this function, by giving precise values of $\depth S/I_{c}(G)^{(k)}$ for $k=2,3$ and $4$, and establish strong comparison inequalities when the exponents are
sufficiently far apart, more generally for arbitrary squarefree monomial ideals.

We begin with the second symbolic power. First, we characterize graphs for which the second symbolic power $I_c(G)^{(2)}$ is componentwise linear.  In this setting, componentwise linearity is equivalent to having linear quotients, and both properties admit the following graph-theoretic characterization. 

\begin{thm-second-power}
Let $G$ be a graph on $[n]$. Then the following conditions are equivalent:
\begin{enumerate}
    \item [\textup{(a)}]
$I_{c}(G)^{(2)}$ is componentwise linear.
     \item [\textup{(b)}]$I_{c}(G)^{(2)}$ has linear quotients.
      \item [\textup{(c)}] $sc(G)\geq3$, or $G\cong K_2\sqcup (n-2)K_1$.
\end{enumerate}
\end{thm-second-power}

This characterization allows us to compute the depth of second symbolic power and yields consequences for the asymptotic behavior of the symbolic depth function; see Theorem \ref{PowerTwo} and Corollary \ref{asymptotic}. The aforementioned example given in~\cite[Proposition 5.5]{fmm2026}, raises the question of which graphs exhibit a strict increase in depth from the third to the fourth symbolic powers. We answer this question by explicitly determining the depths of the third and fourth symbolic powers for arbitrary graphs. To this aim, we study the local cohomology of $S/I_c(G)^{(k)}$ via the reduced simplicial homology of an appropriate simplicial complex, encoding combinatorial information of the graph. 
These results are summarized below.

 \begin{second-third-fourth}
Let $G$ be a graph on the vertex set $[n]$ with $n\geq3$, and let $A$ be the set of isolated vertices of $G$. If $G-A\not\cong K_2$, then 
	\[
	\begin{aligned}
		\depth S/I_c(G)^{(2)}
		&=
		\begin{cases}
n-2, & \text{if each component of } G-A \text{ is a star graph}, \\
n-3, & \text{otherwise}. 
\end{cases}
		\\[6pt]
		\depth S/I_c(G)^{(3)}
		&=
		\begin{cases}
n-2, & \parbox[t]{8cm}{if $G-A$ is a star or a disjoint union of edges}\\ 
n-\max\{3,\operatorname{reg} I(G^c)\}, & \text{otherwise}.
\end{cases}
		\\[6pt]
		\depth S/I_c(G)^{(4)}
		&=
		\begin{cases}
n-2, & \parbox[t]{8cm}{if $G-A$ is a star or a disjoint union of edges,}\\
n-3, & \text{otherwise}.
\end{cases}
	\end{aligned}
	\]   
Furthermore, if $G-A\cong K_2$, then $\depth S/I_c(G)^{(k)}=n-1$, for all $k\geq 1$.
\end{second-third-fourth}

\smallskip

\noindent
Corollary \ref{comparison of 3rd and 4th}  gives  
$$\depth S/I_{c}(G)^{(4)}-\depth S/I_{c}(G)^{(3)}=\max\{\reg I(G^{c})-3,0\},$$ 
for any graph $G$ with $|V(G)|\geq3$. Therefore, for such a graph, depth increases from the third to the fourth symbolic power if and only if $\reg I(G^{c})>3$. This includes the complete multipartite graph example given in~\cite[Proposition 5.5]{fmm2026}, and shows that the same increase occurs for a wider class of graphs. 
The formulas also show that the depths of the second and fourth symbolic powers of complementary edge ideals do not depend on the characteristic of the base field, while the depth of the third symbolic power and depth increase from the third to the fourth symbolic power may depend on it. 

The asymptotic behavior and monotonicity properties of the depth of symbolic powers of squarefree monomial ideals have been studied extensively; see, for instance, \cite{hktt2017,nt2019,mt2025}. The existence of squarefree monomial ideals whose symbolic depth functions are not non-increasing was shown in~\cite{nt2019, mt2025}, while it remains an open problem whether the symbolic depth function of an edge ideal is always non-increasing. For edge ideals, it is also conjectured that the depth of every symbolic powers is at least that of the corresponding ordinary power, and this is known for the second and third powers \cite{m2026}. For cover ideals, the symbolic depth function is known to be non-increasing \cite[Theorem 3.2]{hktt2017}. As it was mentioned above, the analogous property fails for complementary edge ideals. Motivated by these results, we study comparisons between depths of different symbolic powers of arbitrary squarefree monomial ideals. Our approach is based on the modules $K_{M}^{i}=\Hom_{K}(H_{\mathfrak{m}}^{i}(M),K),$ the $K$-dual of the $i$-th local cohomology module of a finitely generated graded $S$-module $M$, whose dimensions provide a convenient framework for comparing the depth of symbolic powers.
Our main result establishes a comparison between the dimension of these modules, from which we derive corresponding inequalities for both Serre depth and depth. 

It is known by the local duality theorem (see, e.g. \cite[Theorem 3.6.19]{BHBook}) that
$K_{M}^{i}\cong\Ext_{S}^{n-i}(M,\omega_{S})$, 
where $\omega_{S}$ is the canonical module of $S$. 
Recall that the {\it big height }of an ideal $I$, denoted by $\bight I$ is the maximum height of a minimal prime ideal of $I$. Then we prove the following result. 

\begin{theorem-ineq of dual mod}
Let $I$ be a squarefree monomial ideal. Then for any $t\geq 1$ we have
$$\dim K_{S/I^{(t)}}^{i}\leq\dim K_{S/I^{(k)}}^{i}$$
for all $i\geq0$ and $k\geq(\bight I-1)t-(\bight I-2)$.
\end{theorem-ineq of dual mod}

This theorem provides corresponding comparisons for Serre depth $(S_{r})$ of symbolic powers of squarefree monomial ideals, which yields an extension and generalization of \cite[Theorem 2.7]{nt2019}. Let us recall the definition of Serre depth for $(S_{r})$, which was introduced in \cite{ppty2022,ppty2023} for finitely generated modules over polynomial rings and subsequently extended in \cite{fSerre2026} to finitely generated modules over Noetherian local rings and standard graded algebras over a field.

Let $S = K[x_1,\dots,x_n]$ be the polynomial ring, and let $M$ be a finitely generated graded $S$-module. 
The \emph{Serre depth}  for $(S_r)$ of $M$ is defined as
$$S_{r}\text{-}\!\operatorname{depth} M = \min\big\{\, j \; : \; \dim K_M^j \geq j - r + 1 \,\big\}.$$    
As a consequence of Theorem \ref{ineq of dual mod} we obtain

\begin{cor-Serre depth Decrease}
Let $I$ be a squarefree monomial ideal, and let $r\geq2$. Then for any $t\geq 1$ we have 
$$S_{r}\mbox{-}\depth S/I^{(t)}\geq S_{r}\mbox{-}\depth S/I^{(k)} \mbox{ for any } k\geq (\bight I-1)t-(\bight I-2).$$
In particular, we have 
$$\depth S/I^{(t)}\geq\depth S/I^{(k)} \mbox{ for any } k\geq (\bight I-1)t-(\bight I-2).$$
\end{cor-Serre depth Decrease}

Thus, although the depth function is not necessarily non-increasing, its values still satisfy strong inequalities when the exponents are sufficiently far apart. Combining Corollary \ref{comparison of 3rd and 4th} with the preceding comparison theorem, we obtain the following inequalities for an arbitrary graph $G$ with $|V(G)|\geq3$:
$$\depth S/I_{c}(G)^{(2)}\geq\depth S/I_{c}(G)^{(4)}\geq\depth S/I_{c}(G)^{(3)}\geq\depth S/I_{c}(G)^{(k)}\mbox{ for all }k\geq5.$$

Finally, we return to complementary edge ideals and study how the symbolic depth behaves under multi-whiskering. This operation, introduced in \cite{mpt2025}, generalizes the classical whiskering construction of \cite{v1990}. 
Let $G=G_0[n_1,\ldots,n_h]$ be a multi-whisker graph of a graph $G_{0}$ without isolated vertices, and let $S=K[V(G_0)]$, $T=K[V(G)]$ and $N=\sum_{1\leq i\leq h}n_{i}$. In Theorem \ref{multi-whisker}, we prove that, for every $k\geq1$, 
\[
N+\depth S/I_c(G_0)^{(k)}-2
\leq
\depth T/I_c(G)^{(k)}
\leq
N+\depth S/I_c(G_0)^{(k)}.
\]

Thus, after the natural shift induced by the number of added whiskers, multi-whiskering changes the symbolic depth sequence by at most two. The proof makes use of the theory of $d$-Leray simplicial complexes and the result of Kalai and Meshulam \cite{km2006} on intersections of Leray complexes. Combining with  \cite[Proposition 5.9]{fmm2026}, we obtain whisker graphs whose complementary edge ideals have symbolic depth functions that are not non-increasing. 

\smallskip

In Section \ref{sec5}, we investigate the Serre depth for $(S_{2})$, Serre's condition $(S_{2})$ and the Cohen-Macaulay property of symbolic powers of $I_c(G)$. Before studying the $S_{2}$-depth, in Proposition \ref{S_r-I_c(G)}, we first characterize when $S/I_{c}(G)$ satisfies Serre's condition $(S_{r})$ for each $r\geq3$. Hien and Kumar recently obtain a characterization for $(S_{2})$ in \cite[Theorem 4.2]{hk2026}. Our result extends their characterization to all $r\geq2$. In \cite{mt2025}, Terai and the third-named author initiated a systematic study of Serre depth for Stanley--Reisner rings. In particular, they obtained a combinatorial characterization of the $S_{2}$-depth of symbolic powers of Stanley--Reisner ideals and demonstrated that the resulting sequence need not be non-increasing in general. On the other hand, they proved that, for cover ideals of graphs, the $S_{2}$-depth of symbolic powers is non-increasing. These results naturally raise the question of how the Serre depth for $(S_{2})$ behaves for symbolic powers of complementary edge ideals. We show that, in this setting, the Serre depth for $(S_{2})$ of symbolic powers is not necessarily non-increasing, while we also identify conditions under which monotonicity does hold. 

To establish these results, we first develop a general description of the $S_{2}$-depth of symbolic powers of squarefree monomial ideals. More precisely, in Proposition~\ref{S_2General}, we express the $S_{2}$-depth of the symbolic powers of an arbitrary squarefree monomial ideal in terms of the connectivity of appropriate degree complexes. This extends a result related to the study of symbolic powers of unmixed Stanley--Reisner ideals obtained in \cite[Proposition 3.10]{mt2025}. We then specialize this description to complementary edge ideals, and using auxiliary Lemmata~\ref{S_2depth lemma},  \ref{lemma for S_2-depth}, and \ref{S_2Lemma}, we obtain a combinatorial description of the $S_{2}$-depth, as follows. Here, $H\subseteq G$ means that $G$ contains an induced subgraph isomorphic to $H$. 

\begin{theorem-S_2main}
Let $G$ be a graph on $[n]$ with $n\geq3$, without isolated vertices and $I=I_c(G)$. The following statements hold.
\begin{enumerate}
\item[\textup{(a)}] If $G$ is triangle-free, then
\[
(S_{2}\mbox{-}\depth S/I^{(k)})_{k\geq1}=
\begin{cases}
(n-3,n-3,\ldots), & \text{if } C_4\subseteq G, \\
(n-2,n-3,n-3,\ldots), & \text{if } P_4\subseteq G, \, C_4\nsubseteq G, \\
(n-2,n-2,n-3,n-3,\ldots), & \text{if } P_{3}\sqcup K_{1}\subseteq G,\, C_4\nsubseteq G,\, P_4\nsubseteq G, \\
(n-2,n-2,\ldots), & \text{ otherwise}.
\end{cases}
\]
\item[\textup{(b)}] If $G$ contains a triangle, then
\[
S_{2}\mbox{-}\depth S/I^{(k)}=
\begin{cases}
n-4, &
\text{if } K_{5}-e\subseteq G, 
\text{ and } k=7\mbox{ or }k\geq9,\\ 
n-3, &
\text{otherwise}.
\end{cases}
\]
\end{enumerate}
In particular, the following conditions are equivalent: 
\begin{enumerate}
\item[\textup{(1)}] $K_{5}-e\nsubseteq G$. 
\item[\textup{(2)}] $S_{2}\mbox{-}\depth\, S/I_{c}(G)^{(k)}\geq S_{2}\mbox{-}\depth\, S/I_{c}(G)^{(k+1)}$ for all $k\geq1$.
\end{enumerate}
Moreover, if $G$ contains $K_{5}-e$, then 
$S_{2}\mbox{-}\depth S/I_{c}(G)^{(7)}<S_{2}\mbox{-}\depth S/I_{c}(G)^{(8)}$. 
\end{theorem-S_2main}

In \cite[Theorem 4.5]{hk2026}, Hien and Kumar classified the graphs for which $S/I_{c}(G)^{(k)}$ is Cohen--Macaulay for $k\geq3$. Using our description of the $S_{2}$-depth, we complete the classification by including the second symbolic power. In particular, we show that, for every $k\geq2$, Serre's condition $(S_{2})$ is equivalent to Cohen--Macaulayness. 

\begin{thm-char-of-CM-S_2}
Let $G$ be a graph on the vertex set $[n]$ with $n\geq3$, and let $k\geq2$ be an integer. Then $S/I_{c}(G)^{(k)}$ satisfies Serre's condition $(S_{2})$ if and only if $S/I_{c}(G)^{(k)}$ is Cohen--Macaulay. Moreover, these conditions are equivalent to the following conditions:
\begin{enumerate}
\item[\textup{(1)}] If $k=2$, then either  $G\cong K_{n}$, or $G$ has no isolated vertices and every connected component of $G$ is a star graph, or $G\cong K_{2}\sqcup(n-2)K_{1}$. 
\item[\textup{(2)}] If $k\geq3$, then either $G\cong K_{n}$, or $G$ is a star graph, or $G$ is the disjoint union of edges, or $G\cong K_{2}\sqcup(n-2)K_{1}$. 
\end{enumerate}
In particular, the Cohen--Macaulayness of $S/I_{c}(G)^{(k)}$ does not depend on the characteristic of the base field for all $k\geq1$.  
\end{thm-char-of-CM-S_2}



\section{Preliminaries}\label{sec2}
In this section, we review the necessary background from   commutative algebra and combinatorics that will be used throughout the paper. We refer the reader to \cite{BHBook, HHBook, VBook} for the detailed information on combinatorial commutative algebra.
Throughout, $K$ is a field, $S = K[x_1,\dots,x_n]$ is the polynomial ring in $n$ variables over $K$ with its standard grading,  $\mathfrak{m} = (x_1,\dots,x_n)$ is the graded maximal ideal of $S$, and $G$ is a finite simple graph on $[n]$ with $E(G)\neq \emptyset$.

\subsection{Local cohomology, depth, and regularity}

Let $M$ be a finitely generated graded $S$-module. For $i \geq 0$, we denote by $H^i_{\mathfrak{m}}(M)$ the $i$-th \emph{local cohomology module} of $M$ with support in $\mathfrak{m}$. Local cohomology gives cohomological interpretations of the two central invariants of this paper:
 \[
 \operatorname{depth}(M) = \min\{i : H^i_{\mathfrak{m}}(M) \neq 0\}, \qquad
 \operatorname{reg}(M) = \max\{ i+j : H^i_{\mathfrak{m}}(M)_j \neq 0\}.
 \]
 Equivalently, $\operatorname{depth}(M)$ is the length of a maximal $M$-regular sequence in $\mathfrak{m}$, and the \emph{Castelnuovo--Mumford regularity} $\operatorname{reg}(M)$   can also be computed from a minimal graded free resolution of $M$,
 \[
 \operatorname{reg}(M) = \max\{ j - i : \beta_{i,j}(M) \neq 0\},
 \]
 where $\beta_{i,j}(M) = \dim_K \operatorname{Tor}^S_i(M,K)_j$ are the {\em graded Betti numbers} of $M$.  

\subsection{Stanley–Reisner rings}

A \emph{simplicial complex} $\Delta$ on $[n]=\{1,\ldots,n\}$ is a collection of subsets of $[n]$ closed under taking subsets, i.e.\ $F \in \Delta$ and $F' \subseteq F$ imply $F' \in \Delta$. Elements of $\Delta$ are called \emph{faces}, and maximal faces (under inclusion) are \emph{facets}.
For a face $F \in \Delta$, the \emph{link} of $F$ in $\Delta$ is the simplicial complex
$
 \operatorname{link}_{\Delta} F = \big\{\, G \subseteq [n]\setminus F \; : \; G \cup F \in \Delta \,\big\}.
$
 The \emph{Alexander dual} of $\Delta$ is the simplicial complex
$ \Delta^{\vee} = \big\{\, F \subseteq [n] \; : \; [n]\setminus F \notin \Delta \,\big\}$. The complex $\Delta$ is call {\em connected}, if $\Delta$ cannot be written as a disjoint union $\Delta = \Delta_1 \sqcup \Delta_2$ of two nonempty subcomplexes on disjoint vertex sets. The $i$th-{\em skeleton} of $\Delta$ is the simplicial complex $\Delta^{(i)}=\{F\in \Delta:\, |F|\leq i+1\}$. 
To a simplicial complex $\Delta$ on $[n]$ one associates its \emph{Stanley-Reisner ideal}
 \[
 I_\Delta = \big(x_{i_1}\cdots x_{i_s} : \{i_1,\dots,i_s\} \notin \Delta \big) \subseteq S,
 \]
generated by the monomials supported on the minimal non-faces of $\Delta$. The ring $K[\Delta]=S/I_\Delta$ is called the {\em Stanley–Reisner ring} of $\Delta$. The inverse assignment sends a squarefree monomial ideal $I$ to the \emph{Stanley--Reisner simplicial complex} 
$\Delta(I) = \big\{ F \subseteq [n] \; : \; {\bf x}_F \notin I \big\},$
where ${\bf x}_F = \prod_{i \in F} x_i$. 

For a subset $W \subseteq [n]$, the \emph{restriction} of $\Delta$ to $W$ is the subcomplex $\Delta_W = \{F \in \Delta : F \subseteq W\}$ on the vertex set $W$. Hochster's formula computes the graded Betti numbers of $I_\Delta$ (equivalently, of $S/I_\Delta$) in terms of the reduced simplicial homology of the restrictions of $\Delta$.
 \begin{thm}[Hochster's formula {\normalfont\cite{H1977}}]
  \label{thm:hochster}
 	Let $\Delta$ be a simplicial complex on $[n]$. Then for all $i,j$,
 	\[
 	\beta_{i,j}(S/I_\Delta) = \sum_{W \subseteq [n],\, |W| = j} \dim_K \widetilde{H}_{j-i-1}(\Delta_W; K).
 	\]
 \end{thm}

For $\mathbf{a} = (a_1,\dots,a_n) \in \mathbb{Z}^n$, we set ${\textbf x}^{\mathbf{a}} = x_1^{a_1}\cdots x_n^{a_n}$,
$\operatorname{supp}_+(\mathbf{a}) = \{ i : a_i > 0\}$, and $\operatorname{supp}_-(\mathbf{a}) = \{ i : a_i < 0\}$.
 For a subset $F \subseteq [n]$, let $S_F = S[x_j^{-1} : j \in F]$. The simplicial complex $\Delta_{\mathbf{a}}(I)$ is defined as
 \[
 \Delta_{\mathbf{a}}(I) = \big\{\, F \setminus \operatorname{supp}_-(\mathbf{a}) \;:\; \operatorname{supp}_-(\mathbf{a}) \subseteq F \subseteq [n], \ {\textbf x}^{\mathbf{a}} \notin I S_F \,\big\}.
 \]

Let $I \subseteq S$ be a monomial ideal. Since $S/I$ carries a natural $\mathbb{N}^n$-graded structure inherited from $S$, each local cohomology module $H^i_{\mathfrak{m}}(S/I)$ inherits a $\mathbb{Z}^n$-graded structure. We write $H^i_{\mathfrak{m}}(S/I)_{\mathbf{a}}$ for its graded component in multidegree $\mathbf{a} \in \mathbb{Z}^n$.
In order to compute local cohomology modules of monomial ideals in $S$, one uses a generalized version of Hochster’s formula given by Takayama~\cite{t2005}. 

\begin{thm}[{\normalfont\cite[Theorem 1]{t2005}}]
\label{thm:Takayama}
Let $I\subset S$ be a monomial ideal. Then for every $\mathbf{a} \in \mathbb{Z}^n$ and every $i \geq 0$,
\[
H^i_{\mathfrak{m}}(S/I)_{\mathbf{a}}\cong	\widetilde{H}^{i - |\supp_-({\mathbf{a}})| - 1}\big(\Delta_{\mathbf{a}}(I);K\big).
\]
\end{thm}

For a hypergraph $\mathcal{H}$ with the vertex set $V(\mathcal{H})$ and the edge set $E(\mathcal{H})$, the {\em independence complex} of $\mathcal{H}$ is the simplicial complex $\Delta(\mathcal{H})=\{F\subseteq V(\mathcal{H}):\; e\nsubseteq F, \textit{ for any } e\in E(\mathcal{H})\}$, and the {\em edge ideal} of $\mathcal{H}$ is defined as $I(\mathcal{H})=({\bf x}_e: \, e\in E(\mathcal{H}))$. It is easily seen that $\Delta_{I(\mathcal{H})}=\Delta(\mathcal{H})$.

For a monomial ideal $I$, we denote by $\mathcal{G}(I)$ the set of minimal monomial generators of $I$. Moreover, for any homogeneous ideal $I$ the {\em initial degree} of $I$, denoted by $\indeg I$, is the minimum integer $d$ such that $I_d\neq 0$.
  
\subsection{Symbolic powers and complementary edge ideals} 
 Let $I \subseteq S$ be an ideal, and let $\operatorname{Ass}(I)$ denote the set of its associated primes. For $k\in \mathbb{N}$, the $k$-th \emph{symbolic power} of $I$ is defined to be
 	\[
 	I^{(k)} = \bigcap_{P \in \operatorname{Ass}(I)} \big(I^k S_P \cap S\big).
 	\]

For a squarefree monomial ideal $I$ with the set of minimal prime ideals $\operatorname{Min}(I) = \{P_1,\dots,P_r\}$, the $k$-th symbolic powers of $I$ admits the explicit description $I^{(k)} = P_1^{k} \cap \cdots \cap P_r^{k}$.

The following lemma gives a description of the complex $\Delta_{\bf a}(I_{\Delta}^{(k)})$, and allows us to investigate $\reg S/I_{\Delta}^{(k)}$ via Takayama’s formula. 

\begin{lemma}[{\normalfont\cite[Lemma 1.3]{mt2011}}]
\label{deg cpx}
Let $\Delta$ be a simplicial complex on $[n]$, let $k\geq1$, and let
${\bf a}\in\mathbb N^n$. Then
$$\Delta_{\bf a}(I_{\Delta}^{(k)})=\langle F\in\mathcal F(\Delta):\sum_{i\notin F}a_i\leq k-1\rangle.$$
\end{lemma}

Let $G$ be a finite simple graph with vertex set $V(G) = [n] = \{1,\dots,n\}$ and edge set $E(G)$.  We write $G^c$ for the complement of $G$, i.e.\ the graph on $[n]$ with $\{i,j\} \in E(G^c)$ if and only if $\{i,j\} \notin E(G)$.
The {\em clique complex} of $G$, denoted by  $\operatorname{Cl}(G)$ is the simplicial complex whose faces are cliques (complete subgraphs) of $G$. The {\em edge ideal} of $G$ is defined as $I(G)=(x_ix_j:\, \{i,j\}\in E(G))$, and the {\em cover ideal} of $G$ is 
defined as the Alexander dual of $I(G)$ that is $J(G)=\bigcap_{\{i,j\}\in E(G)} (x_i,x_j)$. 

For a graph $G$ on the vertex set $[n]$ with $E(G)\neq \emptyset$, the \emph{complementary edge ideal} of $G$ is the squarefree monomial ideal
$$I_c(G) = \big({\textbf x}_{[n]\setminus\{i,j\}} : \{i,j\} \in E(G)\big) \subseteq S.$$
 
Let $\mathcal{T}(G)$ denote the set of all triangles ($K_3$-subgraphs) of $G$. The following result determines the minimal primary decomposition of a complementary edge ideal. 

\begin{thm}[{\normalfont\cite[Theorem 2.1]{fm2025} or \cite[Theorem 1.1]{hqm2026}}]\label{primay decomp}
\label{min prime}
Let $G$ be a finite simple graph without isolated vertices. Then 
$$I_{c}(G)=(\bigcap_{e\in E(G^{c})}P_{e})\cap(\bigcap_{T\in\mathcal{T}(G)}P_{T}),$$
where $P_C=(x_i:\, i\in C)$ for any $C\subseteq [n]$. 
\end{thm}
Notice that if $G$ has isolated vertices, then $I_c(G)={\bf x}_A I_c(H)$, where $A$ is the set of isolated vertices of $G$ and $H=G-A$.

Theorem~\ref{min prime} yields that for a graph $G$ without isolated vertices

\begin{equation}\label{sym-comp}
I_{c}(G)^{(k)}=(\bigcap_{e\in E(G^{c})}P_{e}^{k})\cap(\bigcap_{T\in\mathcal{T}(G)}P_{T}^{k}).
\end{equation}

\subsection{Serre depth}
Let $R$ be a Noetherian ring, $M$ a finitely generated $R$-module, and $r \geq 0$ an integer. We say $M$ satisfies \emph{Serre's condition $(S_r)$} if
$\operatorname{depth} M_{\mathfrak{p}}  \geq \min\{r, \dim M_{\mathfrak{p}}\}$ for every prime ideal $\mathfrak{p} \in \operatorname{Supp}(M)$.

Let $S = K[x_1,\dots,x_n]$ be the polynomial ring, and let $M$ be a finitely generated graded $S$-module. For $j \geq 0$, set
$K_M^j = \operatorname{Hom}_K\big(H^j_{\mathfrak{m}}(M), K\big)$,
the $K$-dual of the $j$-th local cohomology module of $M$, where we set $K_{M}^{i}=0$ for all $i>n$. We use the convention $\dim 0=-\infty$. 
The \emph{Serre depth}  for $(S_r)$ of $M$ is defined as
	\[
	S_r\text{-}\!\operatorname{depth} M = \min\big\{\, j \; : \; \dim K_M^j \geq j - r + 1 \,\big\}.
	\]    
 

\section{regularity of symbolic powers of $I_{c}(G)$}\label{sec3}

In this section, we study the Castelnuovo--Mumford regularity of $I_{c}(G)^{(k)}$ and $I_{c}(G)^{(k)}/I_c(G)^k$ and give a combinatorial characterization of the linear resolution property of $I_{c}(G)^{(k)}$. First, we describe multi-graded local cohomology modules of $I_{c}(G)^{(k)}$ in terms of the reduced homology and multi-graded Betti numbers of an associated hypergraph. 

\smallskip

To a graph $G$ on the vertex set $[n]$ without isolated vertices,  an integer $k\geq1$ and a vector ${\bf a}\in\mathbb{Z}^{n}$, we associate a hypergraph $\mathcal{H}_{\bf a}^{(k)}(G)$, as follows. Set $U=[n]\setminus\supp_{-}{\bf a}$. Define the hypergraph $\mathcal{H}_{\bf a}^{(k)}(G)$ on the vertex set $U$, whose edge set is
$$E(\mathcal{H}_{\bf a}^{(k)}(G))=\{\{i,j\}\notin E(G)\,\,: a_{i}+a_{j}\leq k-1\}\cup\{\{i,j,\ell\}\in\mathcal{T}(G)\,\,: a_{i}+a_{j}+a_{\ell}\leq k-1\}.$$

Under this notation, we have the following isomorphism, using Alexander duality theorem (see, for example, \cite[Theorem 1.1]{bt2009}) and Takayama's formula (\cite[Theorem 1]{t2005}): 

\begin{lemma}\label{hyperlemma}
With the notation above, we have 
\[
H_{\mathfrak{m}}^{i}(S/I_{c}(G)^{(k)})_{\bf a}\cong
\begin{cases}
\widetilde{H}^{n-i-2}(\Delta(\mathcal{H}_{\bf a}^{(k)}(G));K), & \text{if } U\neq\emptyset, \\
0, & \text{ otherwise}.
\end{cases}
\]

\noindent In particular, if $U\neq\emptyset$, then we have $$\dim_{K}H_{\mathfrak{m}}^{i}(S/I_{c}(G)^{(k)})_{\bf a}=\beta_{i-n+|U|,\, {\bf x}_{U}}(I(\mathcal{H}_{\bf a}^{(k)}(G))).$$
\end{lemma}
\begin{proof}
From the proof of \cite[Theorem 3.2]{kmt2025}, we have $$\Delta_{\bf a}(I_{c}(G)^{(k)})=\Delta_{\bf a^{+}}(I_{\link_{\Delta(I_{c}(G))}(\supp_{-}{\bf a})}^{(k)}).$$Hence, by Lemma \ref{deg cpx}, we see that $\Delta_{\bf a}(I_{c}(G)^{(k)})=\langle U\setminus C\,\,: C\in E(\mathcal{H}_{\bf a}^{(k)}(G))\rangle=(\Delta(\mathcal{H}_{\bf a}^{(k)}(G)))^{\vee}$. By the Alexander duality theorem, we have $$\widetilde{H}_{r}(\Delta_{\bf a}(I_{c}(G)^{(k)});K)\cong\widetilde{H}^{|U|-r-3}(\Delta(\mathcal{H}_{\bf a}^{(k)}(G);K))\mbox{ for all }r.$$ Thus from Takayama's formula, we obtain 
$$H_{\mathfrak{m}}^{i}(S/I_{c}(G)^{(k)})_{\bf a}\cong\widetilde{H}_{i-|\supp_{-}{\bf a}|-1}(\Delta_{\bf a}(I_{c}(G)^{(k)});K)\cong\widetilde{H}^{n-i-2}(\Delta(\mathcal{H}_{\bf a}^{(k)}(G);K))\mbox{ for all }i.$$The latter part follows from
Hochster's formula.
\end{proof}

For a graph $G$, we set $c(G)$ to be the number of connected components of $G$ which are not isolated vertices, and $$sc(G)=\min\{|V(C)|\,\,: C\text{ is a connected component of $G$ with at least two vertices}\}.$$ In the following we  determine the regularity of symbolic powers of complementary edge ideals, explicitly.  
\begin{thm}\label{reg-equality}
Let $G$ be a graph on the vertex set $[n]$. For each $k\geq1$, we have 
\[
\reg I_{c}(G)^{(k)}=
\begin{cases}
(n-2)k, & \text{if } c(G)=1, \\
\max\{(n-2)k, (n-sc(G))k+sc(G)-1\}, & \text{ if } c(G)>1.
\end{cases}
\]
\end{thm}
\begin{proof}
 By \cite[Proposition 6.3]{fmpoly2025}, and the fact that 
$I_c(G)={\bf x}_AI_c(H)$, where $A$ is the set of isolated vertices of $G$ and $H=G-A$, we may assume that $G$ has no isolated vertices. Fix $k\geq1$ and let $I=I_{c}(G)$. 
Suppose that $H_{\mathfrak{m}}^{i}(S/I^{(k)})_{\bf a}\neq 0$. Set $U=[n]\setminus\supp_{-}{\bf a}$,  $p=i-n+|U|$ and $\mathcal{H}_{\bf a}^{(k)}=\mathcal{H}_{\bf a}^{(k)}(G)$. By Lemma \ref{hyperlemma}, we have $\beta_{p,{\bf x}_{U}}(I(\mathcal{H}_{\bf a}^{(k)}))\neq 0$. Consider the Taylor resolution of $I(\mathcal{H}_{\bf a}^{(k)})$ with respect to the generating set $\{{\bf x}_{F}\,\,: F\in E(\mathcal{H}_{\bf a}^{(k)})\}$ as defined in \cite{t1966}. Since $\beta_{p,{\bf x}_{U}}(I(\mathcal{H}_{\bf a}^{(k)}))\neq 0$, there exist edges $F_{0},\ldots, F_{p}\in E(\mathcal{H}_{\bf a}^{(k)})$ such that $F_{0}\cup\cdots\cup F_{p}=U$. By the definition of $E(\mathcal{H}_{\bf a}^{(k)})$, we have $\sum_{j\in F_{i}}a_{j}\leq k-1$, which implies $$\displaystyle\sum_{j\in U}a_{j}\leq\displaystyle\sum_{0\leq i\leq p}\displaystyle\sum_{j\in F_{i}}a_{j}\leq (p+1)(k-1).$$Since $[n]=U\cup\supp_{-}{\bf a}$ and $p=i-n+|U|$, we get $i=p+|\supp_{-}{\bf a}|$. Notice that $\displaystyle\sum_{j\in\supp_{-}{\bf a}}(a_{j}+1)\leq 0$ by the definition of $\supp_{-}{\bf a}$. 
Hence, we obtain

\begin{equation}\label{|a|+i}
|{\bf a}|+i=\displaystyle\sum_{1\leq j\leq n}a_{j}+i=\displaystyle\sum_{j\in U}a_{j}+\displaystyle\sum_{j\in\supp_{-}{\bf a}}a_{j}+p+|\supp_{-}{\bf a}|\leq\displaystyle\sum_{j\in U}a_{j}+p\leq k(p+1)-1. 
\end{equation}
On the other hand, from Hochster's formula, we have $$\beta_{p,{\bf x}_{U}}(I(\mathcal{H}_{\bf a}^{(k)}))=\dim_{K}\widetilde{H}_{|U|-p-2}(\Delta(\mathcal{H}_{\bf a}^{(k)});K).$$ So  from $\beta_{p,{\bf x}_{U}}(I(\mathcal{H}_{\bf a}))\neq 0$ we obtain $p\leq|U|-2$. 

\smallskip

{\bf Case 1.} Suppose that $c(G)=1$.Then by \cite[Corollary 3.2]{fmStanley-Reisner2025} and \cite[Theorem 1.2]{fmpoly2025} we have $\reg I^{(k)}\geq \reg I^k=k(n-2)$. 
So it is enough to prove $\reg I^{(k)}\leq k(n-2)$.  
We show that $p+1\leq n-2$. 
We have $p\leq|U|-2$, which implies that $p+1\leq |U|-1$. So if $U\neq [n]$, then $p+1\leq n-2$. Hence, we may assume that $U=[n]$. It is enough to consider the case $p=n-2$. Then, from Hochster's formula, we obtain $0\neq\beta_{n-2,{\bf x}_{[n]}}(I(\mathcal{H}_{\bf a}^{(k)}))=\dim_{K}\widetilde{H}_{n-(n-2)-2}(\Delta(\mathcal{H}_{\bf a}^{(k)});K)$, which implies $\Delta(\mathcal{H}_{\bf a}^{(k)})$ is disconnected. Now, since $\{u,v\}\in E(G)$ is a face of $\Delta(\mathcal{H}_{\bf a}^{(k)})$, by the assumption that $G$ is connected, $\Delta(\mathcal{H}_{\bf a}^{(k)})$ is connected, a contradiction. Hence, we get $p\neq n-2$, and thus $p+1\leq n-2$. This together with (\ref{|a|+i}) implies that
$$|{\bf a}|+i\leq k(p+1)-1\leq k(n-2)-1.$$
Thus, we have shown that if $H_{\mathfrak{m}}^{i}(S/I^{(k)})_{\bf a}\neq 0$, then we have $|{\bf a}|+i\leq k(n-2)-1$, which implies that $\reg(S/I^{(k)})\leq k(n-2)-1$, as required. 

\smallskip

{\bf Case 2.} Suppose that $c(G)>1$. We have $p\leq |U|-2$. If $U\neq[n]$, then we have $p+1\leq |U|-1\leq n-2$. So by (\ref{|a|+i}) we get $|{\bf a}|+i\leq k(p+1)-1\leq k(n-2)-1$, which implies that $\reg I^{(k)}\leq k(n-2)$.
Similarly, if $p\leq n-3$, then $p+1\leq n-2$ and we obtain the inequality  $\reg I^{(k)}\leq k(n-2)$. So we may assume that
$U=[n]$ and $p=n-2$. In this case, since $\beta_{n-2,{\bf x}_{[n]}}(I(\mathcal{H}_{\bf a}))\neq0$, 
by Lemma \ref{hyperlemma}, we get $\widetilde{H}_{0}(\Delta(\mathcal{H}_{\bf a}^{(k)});K)\neq0$. Hence $\Delta(\mathcal{H}_{\bf a}^{(k)})$ is disconnected. Notice that $G$ is a spanning subgraph of the 1-skeleton of $\Delta(\mathcal{H}_{\bf a}^{(k)})$. Take a connected component of $\Delta(\mathcal{H}_{\bf a}^{(k)})$ with the minimum number of vertices, and let $W$ be its vertex set. Then, $W$ is a union of vertex sets of some connected components of $G$.
Notice that $|W|\leq n-|W|$ and $sc(G)\leq |W|$. For any $u\in W$, $v\in [n]\setminus W$, we have $\{u,v\}\notin\Delta(\mathcal{H}_{{\bf a}}^{(k)})$, which means $\{u,v\}\in E(\mathcal{H}_{\bf a}^{(k)})$. Hence, $a_{u}+a_{v}\leq k-1$. Set $t=\max\{a_{u}\,\,: u\in W\}$. Then we have $a_{v}\leq (k-1)-t$ for each $v\in[n]\setminus W$. Hence, from $n-|W|\geq |W|$ and $t\geq0$, we get 
$$\displaystyle\sum_{1\leq j\leq n}a_{j}\leq |W|t+(n-|W|)(k-1-t)=(n-|W|)(k-1)-(n-2|W|)t\leq (n-|W|)(k-1).$$Therefore, since $i=p=n-2$ and $sc(G)\leq |W|$, we obtain $$|{\bf a}|+i\leq\ (n-sc(G))(k-1)+n-2=(n-sc(G))k+sc(G)-2,$$ which implies
$\reg I^{(k)}\leq (n-sc(G))k+sc(G)-1$. So we have shown that
$$\reg I^{(k)}\leq\max\{(n-2)k, (n-sc(G))k+sc(G)-1\}.$$ 

We finally prove the converse inequality. By \cite[Lemma 3.1]{SF} we have  $\reg I^{(k)}\geq\deg(I^{(k)})\geq k\deg(I)=(n-2)k$ for all $k$. Hence, it is enough to prove that $\reg I^{(k)}\geq (n-sc(G))k+sc(G)-1$. Take a connected component $C$ of $G$ such that $|V(C)|=sc(G)$. Define a vector ${\bf a}$ by $a_{j}=0$ if $j\in V(C)$ and $a_{j}=k-1$ otherwise. Then we have $U=[n]$. For any vertices $u\in V(C)$ and $v\in [n]\setminus V(C)$, we have $\{u,v\}\notin E(G)$. Moreover, since $a_{u}+a_{v}=0+(k-1)=k-1$, we have $\{u,v\}\in E(\mathcal{H}_{\bf a}^{(k)})$. Hence $\Delta(\mathcal{H}_{\bf a}^{(k)})$ is disconnected and thus, by Lemma \ref{hyperlemma}, we get $H_{\mathfrak{m}}^{n-2}(S/I^{(k)})_{\bf a}\neq0$. Therefore, we obtain $\reg S/I^{(k)}\geq |{\bf a}|+n-2$. By the definition of the vector $\bf a$, we have $|{\bf a}|=(n-sc(G))(k-1)$. Thus we get $$\reg S/I^{(k)}\geq (n-sc(G))(k-1)+(n-2)=(n-sc(G))k+sc(G)-2,$$ as required. 
\end{proof}

Combining Theorem \ref{reg-equality} with \cite[Theorem 4.1]{fm2025}, we can compare the regularities of symbolic powers and ordinary powers of $I_{c}(G)$, which gives a negative answer to \cite[Question 4.11]{fm2025}. 

\begin{cor}\label{reg-compare}
Let $G$ be a graph on the vertex set $[n]$. Then we have
\[
\reg I_{c}(G)^{(k)}=
\begin{cases}
(n-2)k, & \text{if } c(G)=1 \text { or }c(G)>1, sc(G)\geq3 \text{ and } k\geq2, \\
(n-2)k+1, & \text{ if } c(G)>1 \text{ and }sc(G)=2, \\
n-1, & \text{ if } c(G)>1, sc(G)\geq3 \text{ and } k=1. 
\end{cases}
\]
In particular, the following statements hold: 
\begin{enumerate}
\item [\textup{(a)}]  $\reg I_{c}(G)^{(k)}\leq \reg I_{c}(G)^k$ for all $k\geq 1$.
\item [\textup{(b)}]  The following are equivalent:
\begin{enumerate}
\item[\textup{(1)}]  $\reg I_{c}(G)^{(k)}=\reg I_{c}(G)^{k}$  for some $k\geq2$, 
\item[\textup{(2)}]  $\reg I_{c}(G)^{(k)}=\reg I_{c}(G)^{k}$ for all $k\geq1$, 
\item[\textup{(3)}]  $c(G)=1$ or $(c(G), sc(G))=(2,2)$.
\end{enumerate}
\item[\textup{(c)}]  $(\reg I_{c}(G)^{(k)}/k)_{k\geq1}$ is non-increasing.  
\item [\textup{(d)}]  $\lim_{k\rightarrow\infty}\dfrac{\reg I_{c}(G)^{(k)}}{k}=n-2.$
\end{enumerate}
\end{cor}

As a consequence of Corollary \ref{reg-compare}, we obtain the following classification of the symbolic powers of complementary edge ideals that have a linear resolution. 

\begin{cor}\label{linear resol}
Let $G$ be a graph on $[n]$ with $n\geq3$, and let $k\geq1$. Then $I_{c}(G)^{(k)}$ has a linear resolution if and only if one of the following holds: 
\begin{enumerate}
\item[\textup{(1)}] $k=1$ and $c(G)=1$,
\item[\textup{(2)}]  $k\geq2$, and $G$ is one of the graphs $K_{2}\sqcup \ell K_{1}$,$ P_{3}\sqcup mK_1$, $C_{3}\sqcup mK_1$, $ P_{4}\sqcup mK_1$ and $C_{4}\sqcup mK_1$ for some $\ell\geq1$ and $m\geq 0$.  
 
\end{enumerate}
\end{cor}
\begin{proof}
If $k=1$, then, by \cite[Theorem 2.2]{hqm2026} or \cite[Corollary 4.8]{fm2025}, $I_c(G)$ has a linear resolution if and only if $c(G)=1$. Suppose $k\geq2$.
Let $A$ be the set of isolated vertices of $G$, and $H=G-A$. Suppose that $H= K_{2}$. Then, we have $G\cong K_{2}\sqcup \ell K_{1}$, where $\ell=|A|\geq 1$, as $n\geq 3$. In this case, $I_{c}(G)=({\bf x}_{A})$ is a principal ideal, and hence, $I_{c}(G)^{(k)}$ has a linear resolution. Thus, we may assume that $H\not\cong K_{2}$. Then $I_c(G)={\bf x}_AI_c(H)$ and by \cite[Proposition 6.3]{fmpoly2025} we have $I_c(G)^{(k)}={\bf x}_A^kI_c(H)^{(k)}$. So $I_c(G)^{(k)}$ has a linear resolution if and only if $I_c(H)^{(k)}$ has a linear resolution. Since $G=H\sqcup mK_1$, where $m=|A|$, we need to show that $I_c(H)^{(k)}$ has a linear resolution if and only if $H$ is one of the graphs $P_3,C_3,P_4$ and $C_4$. Let  $V(H)=[n]$ and $I=I_{c}(H)$. 
First, we suppose $n\geq5$. Notice that ${\bf x}_{[n]}\in I^{(2)}$. Suppose that $k=2p$ for some $p$. Then, we have ${\bf x}_{[n]}^{p}\in I^{(2p)}$. Since $\deg {\bf x}_{[n]}^{p}=np$, by the assumption that $n\geq5$, we obtain $np<2p(n-2)=(n-2)k$, which implies $\indeg I^{(k)}<(n-2)k$. Next, suppose that $k=2p+1$ for some $p$. Then, we have ${\bf x}_{[n]}^{p}f\in I^{(2p)}I\subset I^{(2p+1)}=I^{(k)}$ for any minimal monomial generator $f$ of $I$. Since $\deg {\bf x}_{[n]}^{p}f=np+(n-2)$, by the assumption that $n\geq5$, we obtain $np+n-2<2p(n-2)+n-2=(n-2)k$, which implies $\indeg I^{(k)}<(n-2)k$. Therefore, if $k\geq2$ and $n\geq5$, then $I^{(k)}$ does not have a linear resolution by Corollary \ref{reg-compare}. Hence we may assume that $n\leq4$. Suppose $n=3$. In this case, $H$ must be $P_{3}$ or $C_{3}$ since $H$ has no isolated vertices. In either case, $I$ is a monomial prime ideal, and hence, we see that $I^{(k)}=I^{k}$, by \cite[Corollary 4.8]{fm2025}, $I^{(k)}$ has a linear resolution. We now assume that $n=4$. Suppose that $c(H)=1$ and $H\not\cong P_{4}, C_{4}$. Then $H\cong K_{1,3}$ or $H$ has triangles. First, suppose $H\cong K_{1,3}$. We may assume that $V(H)=[4]$ and $E(H)=\{\{1,2\},\{1,3\},\{1,4\}\}$. If $k=2p$ for some $p$, then, we can verify $x_{2}^{p}x_{3}^{p}x_{4}^{p}\in I^{(k)}$, and hence $\indeg I^{(k)}\leq3p<2k=\reg I^{(k)}$. Also, if $k=2p+1$ for some $p$, then we can verify $x_{2}^{p}x_{3}^{p+1}x_{4}^{p+1}\in I^{(k)}$, and hence $\indeg I^{(k)}\leq3p+2<2k=\reg I^{(k)}$. Next, suppose that $H$ has  triangles. We may assume that $V(H)=[4]$ and $\{1,2,3\}\in\mathcal{T}(H)$. If $N_H(4)=\{1\}$, then we can verify $x_{1}^{k-2}x_{2}x_{3}x_{4}^{k-1}\in I^{(k)}$, and hence $\indeg I^{(k)}\leq2k-1<2k=\reg I^{(k)}$. Also, if $N_H(4)=\{1,2\}$, then we can verify $x_{1}^{k-1}x_{3}x_{4}^{k-1}\in I^{(k)}$, and hence $\indeg I^{(k)}\leq2k-1<2k=\reg I^{(k)}$. Moreover, if $N_H(4)=\{1,2,3\}$, then we can verify $x_{1}^{k-1}x_{2}^{k-2}x_{3}x_{4}\in I^{(k)}$, and hence $\indeg I^{(k)}\leq2k-1<2k=\reg I^{(k)}$. Therefore, if $c(H)=1$ and $H\not\cong P_{4}, C_{4}$, then $\indeg I^{(k)}<\reg I^{(k)}$, that is, $I^{(k)}$ does not have a linear resolution. Suppose that $c(H)>1$ and $n=4$. Then $H\cong K_{2}\sqcup K_{2}$. In this case, by Corollary \ref{reg-compare}, we get $\reg I^{(k)}=2k+1$. Since $\indeg I^{(k)}\leq\indeg I^{k}=2k<2k+1=\reg I^{(k)}$, $I^{(k)}$ does not have a linear resolution. Finally, suppose that $H\cong P_{4}$ or $C_{4}$. We may assume that $\{1,3\}, \{2,4\}\notin E(H)$. Let ${\bf x}^{\bf a}=x_{1}^{a_{1}}x_{2}^{a_{2}}x_{3}^{a_{3}}x_{4}^{a_{4}}\in I^{(k)}$. Then, we have $a_{1}+a_{3}\geq k$ and $a_{2}+a_{4}\geq k$, which implies that $\deg {\bf x}^{\bf a}\geq 2k$. Hence, $\indeg I^{(k)}\geq 2k$. Also, we have $\indeg I^{(k)}\leq\indeg I^{k}=2k$, and hence $\indeg I^{(k)}=2k=\reg I^{(k)}$ by Corollary \ref{reg-compare}, as required. 
\end{proof}

We conclude the section with an explicit formula for $\reg I_{c}(G)^{(k)}/I_{c}(G)^{k}$ for an arbitrary graph $G$ and all $k\geq2$. To establish the formula, we first prove the following lemma.

\begin{lemma}\label{regI/J}
Let $I$ and $J$ be homogeneous ideals of $S$ with $0\neq J\subsetneq I$. Then
$$\reg I/J \geq\indeg J-1.$$
In particular, if $I^{(k)}\neq I^k$, then $\reg I^{(k)}/I^k\geq k\indeg I-1$ for all $k\geq 1$.
\end{lemma}
\begin{proof}
Take a minimal graded free resolution
$$\cdots\rightarrow \bigoplus_{j}S(-b_{j})\xrightarrow{\delta_{1}}F_{0}\xrightarrow{\epsilon}I/J\rightarrow0.$$
Let $\varphi:F_{0}\rightarrow I$ be a degree zero lift of $\epsilon$. Then we see that $\operatorname{Im}(\varphi\circ\delta_{1})\subset J$. For $0\neq g\in J$, we have $gF_{0}\subset\operatorname{Ker}\epsilon=\operatorname{Im}\delta_{1}$, which implies that $0\neq g\varphi(F_{0})=\varphi(gF_{0})\subset\operatorname{Im}(\varphi\circ\delta_{1})$. Hence, there exists a homogeneous basis element $e_{j}$ of $\bigoplus_{j}S(-b_{j})$ such that $(\varphi\circ\delta_{1})(e_{j})\neq0$. Since $\deg e_{j}=b_{j}$ and $(\varphi\circ\delta_{1})(e_{j})\in J$, we obtain that 
$$\reg I/J\geq b_{j}-1\geq\indeg J-1,$$
which completes the proof. 
\end{proof}

\begin{thm}\label{regQuotient}
Let $G$ be a graph, and let $k$ be a positive integer with $I_{c}(G)^{(k)}/I_{c}(G)^{k}\neq0$. Then we have $$\reg I_{c}(G)^{(k)}/I_{c}(G)^{k}=\reg I_{c}(G)^{k}-1= (n-2)k+\min\{k,c(G)-1\}-1.$$
\end{thm}
\begin{proof}
Set $I=I_c(G)$. 
First, assume that $G$ has no isolated vertices. We claim that 
\begin{equation}\label{ineqReg}
 (n-2)k-1\leq\reg I^{(k)}/I^{k}\leq \reg S/I^{k}.    
\end{equation}
By Lemma~\ref{regI/J} and the fact that $I$ is generated in degree $n-2$, we obtain the left-hand side inequality. So in what follows, we verify the inequality $\reg I^{(k)}/I^{k}\leq \reg S/I^{k}$. From \cite[Theorem 4.1]{fm2025}, we know 
$$\reg S/I^{k}=(n-2)k+\min\{k, c(G)-1\}-1.$$ Since $\height I\geq2$, we see that $\dim I^{(k)}/I^{k}\leq n-3$ and $\dim S/I^{(k)}\leq n-2$. Now, from Corollary \ref{reg-compare} (d), we know $\lim_{k\rightarrow\infty}\reg I^{(k)}/{k}=n-2$. Hence, from \cite[Theorem 2.2]{ht2023}, for all $i\geq0$ and $j>(n-2)k-n+2$, we obtain $H_{\mathfrak{m}}^{i}(S/I^{(k)})_{j}=0$. Consider the natural short exact sequence 
$$0\rightarrow I^{(k)}/I^{k}\rightarrow S/I^{k}\rightarrow S/I^{(k)}\rightarrow0. \quad(*)$$
Take $i,j$ such that $0\leq i\leq n-3$ and $j+i>\reg S/I^{k}$. Then we have $H_{\mathfrak{m}}^{i-1}(S/I^{(k)})_{j}=0$. From the long exact sequence of local cohomology modules induced from $(*)$, we see that $H_{\mathfrak{m}}^{i}(I^{(k)}/I^{k})_{j}=0$. Also, if $i>n-3$, then, since $\dim I^{(k)}/I^{k}\leq n-3$, we also obtain that $H_{\mathfrak{m}}^{i}(I^{(k)}/I^{k})_{j}=0$. Hence, we get $\reg I^{(k)}/I^{k}\leq\reg S/I^{k}$, and the claim is proved. 
Moreover, from graded local duality theorem (see, \cite[Theorem 3.6.19]{BHBook}), we see that $\beta_{i,j}(S/I^{(k)})=0$ for all $i\geq0$ and $j>(n-2)k+2$. 
We distinguish the following cases. 

{\bf Case 1.} Suppose that $c(G)=1$. Then, by (\ref{ineqReg}), 
we clearly have $\reg I^{(k)}/I^{k}=\reg S/I^{k}=(n-2)k-1$. 

{\bf Case 2.} Suppose that $c(G)\geq3$. Since $k\geq2$, we have $\reg S/I^{k}\geq(n-2)k+1$. Then, from the long exact sequence of local cohomology modules induced from $(*)$ and the fact that $\reg S/I^{(k)}\leq (n-2)k$, we obtain that 
$$\reg S/I^{k}\leq\max\{\reg I^{(k)}/I^{k}, \reg S/I^{(k)}\}\leq\max\{\reg I^{(k)}/I^{k}, (n-2)k\},$$
which implies that $\reg I^{(k)}/I^{k}\geq \reg S/I^{k}$. This together with (\ref{ineqReg}) proves the assertion.

{\bf Case 3.} Suppose that $c(G)=2$. We may assume that $G\neq 2K_{2}$. Let $C_{1}$ and $C_{2}$ be connected component of $G$ and set $n_{1}=|V(C_{1})|$ and $n_{2}=|V(C_{2})|$. We may assume that $|V(C_{2})|\geq3$. Now we have 
$I={\bf x}_{C_{2}}J_{1}+{\bf x}_{C_{1}}J_{2}$, where $J_{i}=I_{c}(C_{i})$ for $i=1,2$. From the proof of \cite[Theorem 4.1]{fm2025}, we know $I^{k}=P+Q$ is a Betti splitting and $P\cap Q={\bf x}_{C_{1}}^{k}{\bf x}_{C_{2}}J_{2}^{k-1}$, where $P=\sum_{0\leq t\leq k-1}{\bf x}_{C_{2}}^{k-t}{\bf x}_{C_{1}}^{t}J_{1}^{k-t}J_{2}^{t}$ and $Q={\bf x}_{C_{1}}^{k}J_{2}^{k}$. Therefore, we know 
$$\beta_{i,j}(I^{k})=\beta_{i,j}(P)+\beta_{i,j}(Q)+\beta_{i-1,j}(P\cap Q).$$
Now, from \cite[Theorem 4.1]{fm2025}, we see that $J_{2}^{k-1}$ has an $(n_{2}-2)(k-1)$-linear resolution. Hence, $P\cap Q$ has an $(n-2)k+2$-linear resolution, which implies that $\beta_{1, (n-2)k+3}(P\cap Q)>0$. Therefore, we obtain that 
$$\beta_{3, (n-2)k+3}(S/I^{k})=\beta_{2,(n-2)k+3}(I^{k})\geq\beta_{1,(n-2)k+3}(P\cap Q)>0.$$
Now, from the short exact sequence $(*)$, we obtain the long exact sequence 
$$\cdots\rightarrow\Tor_{3}(I^{(k)}/I^{k})_{(n-2)k+3}\rightarrow\Tor_{3}(S/I^{k})_{(n-2)k+3}\rightarrow\Tor_{3}(S/I^{(k)})_{(n-2)k+3}\rightarrow\cdots.$$
Since $\beta_{3,(n-2)k+3}(S/I^{(k)})=0$, we have $\Tor_{3}(S/I^{(k)})_{(n-2)k+3}=0$. On the other hand, we know $\beta_{3,(n-2)k+3}(S/I^{k})>0$. Hence, $\Tor_{3}(I^{(k)}/I^{k})_{(n-2)k+3}\neq0$, that is, $\beta_{3,(n-2)k+3}(I^{(k)}/I^{k})\neq0$, and thus we see that $$\reg I^{(k)}/I^{k}\geq((n-2)k+3)-3=(n-2)k=\reg S/I^{k},$$ as required.

Now, assume that $G$ has isolated vertices, and let $A$ be the set of isolated vertices of $G$. Then $I={\bf x}_AI_c(H)$, where $H=G-A$ and hence 
$I^{(k)}/I^{k}\cong (I_c(H)^{(k)}/I_c(H)^{k})(-k|A|)$. This implies that 
\[
\begin{aligned}
\reg I^{(k)}/I^{k}
&=\reg I_c(H)^{(k)}/I_c(H)^{k}+k|A|\\
&=\reg I_c(H)^k+k|A|-1\\
&=(n-|A|-2)k+\min\{k,c(H)-1\}+k|A|-1\\
&=(n-2)k+\min\{k,c(G)-1\}-1,
\end{aligned}
\]
which completes the proof. 
\end{proof}


\section{depth of symbolic powers of $I_{c}(G)$}\label{sec4}

In this section, we study the behavior of the symbolic depth function $k\mapsto \depth S/I_c(G)^{(k)}$. We determine its precise value for $k=2,3,4$, establish bounds for arbitrary symbolic powers, and show that, although the symbolic depth function need not be non-increasing, its values satisfy strong comparison inequalities when the exponents are sufficiently far apart. We further investigate how the multi-whiskering operation affects this function.

For an arbitrary graph $G$ with the set of isolated vertices $A$, we have $I_{c}(G)^{(k)}={\bf x}_A^k I_{c}(H)^{(k)}$, where $H=G-A$ (see \cite[Proposition 6.3]{fmpoly2025}). Consequently, if $I_c(H)\neq (1)$, or equivalently $H\not\cong K_2$, then
\begin{equation}\label{depth-isolated}
 \depth S/I_{c}(G)^{(k)}=\depth S/I_{c}(H)^{(k)}=\depth S'/I_{c}(H)^{(k)}+|A|,   
\end{equation}
 where $S'=S/(x_i:\, i\in A)$. We use this fact throughout this section, to reduce to the case that the graph has no isolated vertices.

We first establish a uniform lower bound for the depth of all symbolic powers.

\begin{prop}\label{prop-lowerbound}
Let $G$ be a graph with  $s$ isolated vertices and $|V(G)|-s\geq7$. Then we have $$\depth S/I_{c}(G)^{(k)}\geq s+2\mbox{ for all }k\geq1.$$
\end{prop}
\begin{proof}
Let $A$ be the set of isolated vertices of $G$, let $H=G-A$ and $V(H)=[m]$. Since $m\geq 7$, by (\ref{depth-isolated}), it is enough to show that $\depth S/I_{c}(H)^{(k)}\geq 2$. Fix $k\geq1$ and ${\bf a}\in\mathbb{N}^{m}$. By Lemma \ref{deg cpx}, have $$\Delta_{\bf a}(I_{c}(H)^{(k)})=\langle[m]\setminus F\,:\,  [m]\setminus F \text{ is a facet of }\Delta(I_{c}(H)),\,\,  \sum_{i\in F}a_{i}\leq k-1\rangle.$$ 
Using \cite[Theorem 2.1]{fm2025} we know that $[m]\setminus F$  is a facet of $\Delta(I_{c}(H))$ if and only if $|F|=2$ with
$F\notin E(H)$, or $F$ is a triangle in $H$.
Fix facets $[m]\setminus F$, $[m]\setminus F^{\prime}$ of $\Delta_{\bf a}(I_{c}(H)^{(k)})$. Then we have $([m]\setminus F)\cap([m]\setminus F^{\prime})=[m]\setminus(F\cup F^{\prime})$ and $|F\cup F^{\prime}|\leq|F|+|F^{\prime}|\leq6$. Since $m\geq7$, we obtain $([m]\setminus F)\cap([m]\setminus F^{\prime})\neq\emptyset$, which implies that $\Delta_{\bf a}(I_{c}(H)^{(k)})$ is connected. Thus $\widetilde{H}_{0}(\Delta_{\bf a}(I_{c}(H)^{(k)});K)=0$ Therefore, by Takayama's formula we obtain $\depth S/I_{c}(H)^{(k)}\geq2$ for all $k\geq1$, as required. 
\end{proof}

Recall that a monomial ideal $I\subset S$ has {\em linear quotients} if there exists an ordering $u_1,\ldots,u_m$ of its minimal monomial generators such that, for every $j=2,\ldots,m$, the colon ideal
$(u_1,\ldots,u_{j-1}):u_j$
is generated by a subset of the variables. We put $\set(u_j) = \{i :\,  x_i \in (u_1,\dots,u_{j-1}):u_j\}$. The ideal $I$ is called {\em componentwise linear}, if $I_{\langle d\rangle}= (f\in I:\, f \textrm{ is homogeneous of degree } d)$ has a linear
resolution for all $d\geq 1$.
The following theorem provides a tool for computing the depth of the second symbolic power of $I_{c}(G)$.

\begin{thm}\label{second-power}
Let $G$ be a graph on $[n]$. The following statements are equivalent:
\begin{enumerate}
    \item [\textup{(a)}]
$I_{c}(G)^{(2)}$ is componentwise linear.
     \item [\textup{(b)}]$I_{c}(G)^{(2)}$ has linear quotients.
      \item [\textup{(c)}]$sc(G)\geq3$, or $G\cong K_2\sqcup (n-2)K_1$.
\end{enumerate}
\end{thm}
\begin{proof}
(c)  $\Rightarrow$ (b) Let $A$ be the set of isolated vertices of $G$. If $G\cong K_2\sqcup (n-2)K_1$, then $I_c(G)=({\bf x}_A)$ is a principal ideal and so $I_{c}(G)^{(2)}=({\bf x}_A^2)$ has linear quotients. Now, assume that $sc(G)\geq 3$. Let $H=G-A$. The equality $I_{c}(G)^{(2)}={\bf x}_A^2 I_{c}(H)^{(2)}$ shows that $I_{c}(G)^{(2)}$ has linear quotients, if and only if $I_{c}(H)^{(2)}$ has linear quotients. Also clearly $sc(G)=sc(H)$. Thus we may assume that $G$ has no isolated vertices.
We set $I=I_{c}(G)$, and call a vertex $v\in V(G)$ a cone vertex , if $N_G[v]=V(G)$. From Theorem \ref{min prime}, we see that ${\bf x}^{\bf a}=x_{1}^{a_{1}}\cdots x_{n}^{a_{n}}\in I^{(2)}$ if and only if $a_{i}+a_{j}\geq2$ for each $\{i,j\}\notin E(G)$ and $a_{i}+a_{j}+a_{\ell}\geq2$ for each $\{i,j,\ell\}\in\mathcal{T}(G)$. Set $m_{ij}={\bf x}_{[n]\setminus\{i,j\}}$ for each $\{i,j\}\in E(G)$, $w_{i}={\bf x}_{[n]\setminus N_{G}[i]}$ and $h_{i}={\bf x}_{N_{G}(i)}w_{i}^{2}$. We first prove 
$$\mathcal{G}(I^{(2)})=\{m_{ij}^{2}\,\,: \{i,j\}\in E(G)\}\sqcup\{h_{i}\,\,: \deg_{G}(i)\geq2\}\sqcup F,$$
where $F=\{{\bf x}_{[n]}\}$ if $G$ has no cone vertex, and $F=\emptyset$, otherwise. 
Fix a minimal monomial generator ${\bf x}^{\bf a}=x_{1}^{a_{1}}\cdots x_{n}^{a_{n}}\in\mathcal{G}(I^{(2)})$. We have $a_{i}\in\{0,1,2\}$ for each $i$. Set $A=\{i\,\,: a_{i}=0\}$. It is easily seen that $|A|\leq 2$. First, suppose $|A|=2$ and write $A=\{i,j\}$. By the assumption that $a_{i}+a_{j}\geq2$ for each $\{i,j\}\notin E(G)$, we have $\{i,j\}\in E(G)$. Fix a vertex $\ell$ of $G$ with $\ell\notin\{i,j\}$. If $i,j\in N_{G}(\ell)$, then $\{i,j,\ell\}\in\mathcal{T}(G)$, which implies that $2\leq a_{i}+a_{j}+a_{\ell}=a_{\ell}$. Also, if $i\notin N_{G}(\ell)$ or $j\notin N_{G}(\ell)$, then we have either $2\leq a_{i}+a_{\ell}=a_{\ell}$ or $2\leq a_{j}+a_{\ell}=a_{\ell}$. Hence, we get $a_{\ell}=2$ for all $\ell\notin\{i,j\}$. It follows that  
$${\bf x}^{\bf a}=\prod_{\ell\notin\{i,j\}}x_{\ell}^{2}=({\bf x}_{[n]\setminus\{i,j\}})^{2}=m_{ij}^{2}.$$
Next, suppose $|A|=1$ and write $A=\{i\}$. Fix a vertex $j$ of $G$ with $j\neq i$. If $j\notin N_{G}(i)$, by the assumption that $a_{i}+a_{j}\geq2$, we have $a_{j}=2$. On the other hand, if $j\in N_{G}(i)$, then, by the minimality of ${\bf x}^{\bf a}$, we have $a_{j}=1$. 
Hence, we get ${\bf x}^{\bf a}={\bf x}_{N_{G}(i)}{\bf x}_{[n]\setminus N_{G}[i]}^{2}=h_{i}$. Suppose on the contrary that $\deg_{G}(i)=1$, and let $N_{G}(i)=\{j\}$. Then ${\bf x}^{\bf a}=h_{i}=x_jm_{ij}^{2}$ is divided by $m_{ij}^{2}$, which contradicts to the minimality of ${\bf x}^{\bf a}$. So  $\deg_{G}(i)\geq 2$, as desired. Finally, suppose $|A|=0$. In this case, by the minimality of ${\bf x}^{\bf a}$, we have $a_{i}=1$ for all $i$. Hence, ${\bf x}^{\bf a}={\bf x}_{[n]}$. If there exists a cone vertex $i$, since $w_{i}={\bf x}_{[n]\setminus N_{G}[i]}=1$, $h_{i}={\bf x}_{[n]\setminus\{i\}}$ divides ${\bf x}_{[n]}$, which implies that ${\bf x}_{[n]}$ is not a minimal generator of $\mathcal{G}(I^{(2)})$. 

Conversely, choose an element $u\in \{m_{ij}^{2}\,\,: \{i,j\}\in E(G)\}\cup\{h_{i}\,\,: \deg_{G}(i)\geq2\}\cup F$.
If $u=m_{ij}^{2}$, then it is easy to see from the definition that $u\in \mathcal{G}(I^{(2)})$. Let $u=h_i$ for some $i$ with $\deg_{G}(i)\geq2$, and write $h_i={\bf x}^{\bf a}$. Then $a_i=0$, $a_\ell=2$ for any $\ell\notin N_{G}[i]$ and $a_\ell=1$ for any $\ell\in N_{G}(i)$. Fix $\ell\notin N_{G}[i]$. Then $a_{i}+(a_{\ell}-1)=1<2$. This together with $\{i,\ell\}\notin E(G)$ implies that ${\bf x}^{{\bf a}-{\bf e}_{\ell}}\notin I^{(2)}$. 
Now, consider a vertex $\ell\in N_{G}(i)$.
Since $\deg_{G}(i)\geq2$, there exists $r\in N_{G}(i)$ with $r\neq \ell$.  If $\{\ell,r\}\notin E(G)$, then we have $(a_{\ell}-1)+a_{r}=1<2$, and if if $\{\ell,r\}\in E(G)$, then we have $a_{i}+(a_{\ell}-1)+a_{r}=1<2$. This shows that 
 ${\bf x}^{{\bf a}-{\bf e}_{\ell}}\notin I^{(2)}$. Thus we have shown that $h_{i}\in\mathcal{G}(I^{(2)})$.
Suppose that $G$ has no cone vertex, and $u={\bf x}_{[n]}$. In this case, for each vertex $i$, there exists a vertex $j$ such that $\{i,j\}\notin E(G)$. Then
${\bf x}_{[n]\setminus\{i\}}\notin I^{(2)}$, as $a_i+a_j=1<2$, and hence ${\bf x}_{[n]}$ is a minimal monomial generator of $I^{(2)}$, as desired. We now prove $I^{(2)}$ has linear quotients by considering the following two cases: 

\textbf{Case 1.} Suppose that $G$ has no cone vertex. Consider any total order $>$ on the elements of $\mathcal{G}(I^{(2)})$, satisfying ${\bf x}_{[n]}>h_{i}>m_{j,\ell}^{2}$ for any $i$, $j$ and $\ell$. We show that such order $>$ is an admissible order. To this aim let $u,v\in \mathcal{G}(I^{(2)})$ with $u>v$.
First assume that $u={\bf x}_{[n]}$. If $v=h_i$, then $u:v=x_i$, and we are done. Let $v=m_{j,\ell}^{2}$ for some $\{j,\ell\}\in E(G)$. By the assumption that $sc(G)\geq 3$, for every edge $\{i,j\}\in E(G)$, we have either $\deg_{G}(i)\geq2$ or $\deg_{G}(j)\geq2$.  Let $\deg_G(j)\geq 2$. Then $h_j\in \mathcal{G}(I^{(2)})$ and $h_j:v=x_\ell$. Since $x_\ell| u:v$, we are done in this case. Now, assume that $u=h_i$ and $v=h_j$ for some $i$ and $j$. Then $x_j| u:v$ and ${\bf x}_{[n]}:v=x_j$.  Next, suppose that $u=h_i$ and $v=m_{j,\ell}^{2}$. If $i\in\{j,\ell\}$, say $i=j$, then $u:v=h_j:m_{j,\ell}^{2}=x_\ell$. Now, let  $i\notin\{j,\ell\}$. As discussed above, we may assume that $\deg_G(j)\geq 2$. Then $x_\ell| u:v$ and $h_j:m_{j,\ell}^{2}=x_\ell$. Finally, let $u=m_{i,r}^{2}$ and $v=m_{j,\ell}^{2}$. 
First assume that $\{i,r\}\cap\{j,\ell\}\neq \emptyset$, say $i=j$. Then $\deg_G(i)\geq 2$, $x_\ell| u:v$ and $h_i:v=h_i:m_{i,\ell}^{2}=x_\ell$. Now, suppose $\{i,r\}$ and $\{j,\ell\}$ are disjoint. We may assume that $\deg_G(j)\geq 2$. Then $x_\ell| u:v$ and $h_j:v=x_\ell$. 

\medskip

\textbf{Case 2.} Suppose that $G$ has a cone vertex $s$. Then $h_s={\bf x}_{[n]\setminus \{s\}}$. Let $>$ be a total order on the elements of $\mathcal{G}(I^{(2)})$, satisfying $h_s>h_{i}>m_{j,\ell}^{2}$ for any $i$, $j$ and $\ell$. To show that $>$ is admissible, we consider $u,v\in \mathcal{G}(I^{(2)})$ with $u>v$. If $u=h_s$ and $v=h_i$ for some $i\neq s$, then $u:v=x_i$. Now, let $u=h_s$ and $v=m_{j,\ell}^{2}$. If $s\in\{j,\ell\}$, say $s=j$, then $u:v=h_s:m_{s,\ell}^{2}=x_\ell$. Otherwise, by assumption,  
we may assume that $\deg_G(j)\geq 2$ and hence $h_j\in \mathcal{G}(I^{(2)})$ with $h_j:v=x_\ell$ and $x_\ell|u:v$.   
Now, assume that $u=h_i$ and $v=h_j$ for some $i,j\neq s$. Then $x_j| u:v$ and $h_s:v=x_j$. The cases where $u=h_i$ for some $i\neq s$ and $v=m_{j,\ell}^{2}$, and where $u=m_{i,r}^{2}$ and $v=m_{j,\ell}^{2}$. are treated exactly as in Case 1. This completes the proof. 

\smallskip

(b) $\Rightarrow$ (a) holds for any graded ideal, see \cite[Theorem 8.2.15]{HHBook}. 

(a) $\Rightarrow$ (c) As before, let $A$ be the set of isolated vertices of $G$ and $H=G-A$. The characterization of minimal monomial generators of $I_c(H)^{(2)}$ given in the proof of (c) $\Rightarrow$ (b) shows that the highest degree of a generator in a minimal set of
generators of $I_c(H)^{(2)}$ is equal to $2(|V(H)|-2)$. The equality $I_{c}(G)^{(2)}={\bf x}_A^2 I_{c}(H)^{(2)}$ shows that the corresponding highest degree for $I_c(G)^{(2)}$ is $2|A|+2(|V(H)|-2)=2(n-2)$. Thus \cite[Corollary 8.2.14]{HHBook} gives $\reg I_c(G)^{(2)}=2(n-2)$. Hence, Corollary~\ref{reg-compare} yields that either $c(G)=1$, or $c(G)>1$ and $sc(G)\geq3$. This means that either $sc(G)\geq3$, or $c(G)=1$ and $sc(G)=2$. The latter case means $G\cong K_2\sqcup (n-2)K_1$.
\end{proof}

Using Theorem \ref{second-power}, we completely determine the depth of second symbolic powers of complementary edge ideals.

\begin{thm}\label{PowerTwo}
Let $G$ be a graph on the vertex set $[n]$ with $n\geq 3$, and let $A$ be the set of isolated vertices of $G$. Then we have 
\[
\depth S/I_{c}(G)^{(2)}=
\begin{cases}
n-1, & \text{if $G-A\cong K_2$},\\
n-2, & \text{if each component of } G-A \text{ is a star, and $G-A\not\cong K_2$}, \\
n-3, & \text{otherwise}. 
\end{cases}
\]
\end{thm}
\begin{proof}
Let $H=G-A$. If $H\cong K_2$, then $I_{c}(G)^{(2)}=({\bf x}_A^2)$ and so $\depth S/I_{c}(G)^{(2)}=n-1$, as desired. Now suppose that $H\not\cong K_2$. Then $\depth S/I_{c}(G)^{(2)}=\depth S'/I_{c}(H)^{(2)}+|A|$, and we need to show that 
\[
\depth S'/I_{c}(H)^{(2)}=
\begin{cases}
|V(H)|-2, & \text{if each component of } H \text{ is a star graph}, \\
|V(H)|-3, & \text{otherwise}. 
\end{cases}
\]
Here, $S'=S/(x_i:\, i\in A)$.
First assume that $sc(G)\geq 3$. Then $sc(H)\geq 3$, and so by Theorem~\ref{second-power}, $I_{c}(H)^{(2)}$ has linear quotients.
Consider the notation and the admissible order for $I_{c}(H)^{(2)}$ given in the proof of Theorem~\ref{second-power}. It is easily seen that $\set(h_i)=\{x_i\}$ for all $i$, and $\set(m_{j,\ell}^{2})=\{x_j,x_\ell\}$,  if $\deg_H(j)\geq 2$ and  $\deg_H(\ell)\geq 2$. Otherwise, $|\set(m_{j,\ell}^{2})|=1$. Clearly, $H$ has an edge $\{j,\ell\}$ with  $\deg_H(j)\geq 2$ and  $\deg_H(\ell)\geq 2$ if and only if some component of $H$ is not a star graph. Hence, by \cite[Corollary 8.2.2]{HHBook}, we have $\pd S'/I_{c}(H)^{(2)}=2$ if all components of $H$ are star graphs and $\pd S'/I_{c}(H)^{(2)}=3$, otherwise. Now, the Auslander-Buchsbaum formula gives the desired result.

Now, assume that $sc(G)=2$. Since $H\not\cong K_2$, we have $H=L\sqcup K_2$ for a non-empty graph $L$ without isolated vertices. If $L\cong K_2$, then $H\cong 2K_2$ and $I_c(H)$ is a complete intersection of height $2$. Thus $\depth S'/I_{c}(H)^{(2)}=|V(H)|-2$, and we are done.
So we may assume that $L\not\cong K_2$. Let $V(H)=[m]$, $V(L)=V(H)\setminus \{1,2\}$ and $S''=K[V(L)]$. By induction on the number of vertices of the graph we may assume that 	
\[
\depth S''/I_{c}(L)^{(2)}=
\begin{cases}
	|V(L)|-2, & \text{if each component of } L \text{ is a star graph}, \\
	|V(L)|-3, & \text{otherwise}. 
\end{cases}
\] 
So it is enough to show that $\depth S'/I_{c}(H)^{(2)}=\depth S''/I_{c}(L)^{(2)}+2$.

By the description of $\mathcal{G}(I_{c}(H)^{(2)})$ given in the proof of Theorem~\ref{second-power}, it is straightforward to check that
$$I_{c}(H)^{(2)}=({\bf x}_{[m]},{\bf x}_{[m]\setminus \{1,2\}}^2)+x_1^2x_2^2 I_{c}(L)^{(2)}.$$
Since ${\bf x}_{[m]\setminus \{1,2\}}\in I_{c}(L)^{(2)}$, we have $(I_{c}(H)^{(2)}:x_1^2x_2^2)=I_{c}(L)^{(2)}$. Thus 
$$\depth S'/ (I_{c}(H)^{(2)}:x_1^2x_2^2)=\depth S'/ I_{c}(L)^{(2)}=\depth S''/I_{c}(L)^{(2)}+2\leq m-2.$$
Moreover, $(I_{c}(H)^{(2)},x_1^2x_2^2)=({\bf x}_{[m]\setminus \{1,2\}},x_1x_2)^2$. Since ${\bf x}_{[m]\setminus \{1,2\}},x_1x_2$ forms a regular $S'$-sequence, we obtain that $S'/ (I_{c}(H)^{(2)},x_1^2x_2^2)$ is Cohen-Macaulay, and hence 

$$\depth S'/ (I_{c}(H)^{(2)},x_1^2x_2^2)=\dim S'/ ({\bf x}_{[m]\setminus \{1,2\}},x_1x_2)^2=m-2.$$
Thus $\depth S'/ (I_{c}(H)^{(2)}:x_1^2x_2^2)\leq\depth S'/ (I_{c}(H)^{(2)},x_1^2x_2^2)$. Applying \cite[Theorem 4.3(ii)]{CHHKTT} we conclude that
$$\depth S'/I_{c}(H)^{(2)}=\depth S'/ (I_{c}(H)^{(2)}:x_1^2x_2^2)=\depth S''/I_{c}(L)^{(2)}+2,$$ as desired.
\end{proof}

\begin{cor}\label{asymptotic}
Let $G$ be a graph on the vertex set $[n]$ with $n\geq 3$, and let $A$ be the set of isolated vertices of $G$. The following statements hold. 
\begin{enumerate}
\item [\textup{(a)}] For any $k\geq 1$, we have 
\[
\depth S/I_{c}(G)^{(k)}=
\begin{cases}
n-1, & \text{if $G-A\cong K_2$},\\
n-2, & \text{if $G-A$ is a star, or disjoint union of edges,
and $G-A\not\cong K_2$}.
\end{cases}
\]
\item [\textup{(b)}] If $G-A$ is neither a star graph, nor a disjoint union of edges, then $$\lim_{k\rightarrow\infty}\depth S/I_{c}(G)^{(k)} \leq n-3.$$ Moreover, this bound is sharp.
\end{enumerate}
\end{cor}
\begin{proof}
 (a) If $G-A\cong K_2$, since $n\geq 3$, we have $A\neq \emptyset$. So $I_c(G)^{(k)}=({\bf x}_A^k)$, and the assertion is clear. 
 Now, assume that  $G-A\not\cong K_2$. 
 By (\ref{depth-isolated}), we may reduce to the case that $A=\emptyset$.
 Since $G$ is triangle-free and has no isolated vertices, we have $I_c(G)=J(G^c)$. Thus by \cite[Theorem 3.4]{hktt2017}, we have $$\depth S/I_{c}(G)^{(k)}=\depth S/J(G^c)^{(k)}=n-\nu_0(G^c)-1$$ for all $k\geq 2\nu_0(G^c)-1$, where $\nu_0(G^c)$ denotes the ordered matching number of $G^c$.
First assume that $G=K_{1,n-1}$ is a star graph. Then $G^c$ consists of
$K_{n-1}$ together with an isolated vertex. 
It is easily seen that no two edges form an ordered matching in $K_{n-1}$. Hence, $\nu_0(G^c)=1$. Now, let $G$ be the disjoint union of $s$ edges $\{u_i,v_i\}$ for $1\leq i\leq s$. Suppose by contradiction that two edges $e=\{a_1,b_1\}$ and $e'=\{a_2,b_2\}$  form an ordered matching in $G^c$, where $\{a_1,a_2\}$ is an independent set of $G^c$. Since $G^c=K_{2,\ldots,2}$ is a complete multipartite graph with parts $\{u_i,v_i\}$, we have $\{a_1,a_2\}=\{u_i,v_i\}$ for some $i$. Then $b_1$ lives in a different part and $\{a_2,b_1\}\in E(G^c)$ which contradicts to $e$ and $e'$ being an ordered matching. Hence, $\nu_0(G^c)=1$,
which yields the desired equality.

 (b) First assume that some connected component of $G$ is not a star graph. Then
by Theorem \ref{PowerTwo}, $\depth S/I_{c}(G)^{(2)}=n-3$. Using this and  \cite[Theorem 2.4(1)]{hktt2017} we get  $\ell_s(I_{c}(G))\geq 3$. So by \cite[Theorem 2.4(2)]{hktt2017} we have
 $\depth S/I_{c}(G)^{(k)}=n-\ell_s(I_{c}(G))\leq n-3$ for $k\gg0$.
 
Now, assume that each connected component of $G$ is a star graph. Set $H=G-A$. Since $H$ is triangle-free and has no isolated vertices, we have $I_c(H)=J(H^c)$. Moreover, since $H\not\cong K_2$, we have $\depth S/I_{c}(G)^{(k)}=\depth S'/J(H^c)^{(k)}+|A|$. This together with \cite[Theorem 3.4]{hktt2017}, gives $\lim_{k\rightarrow\infty}\depth S/I_{c}(G)^{(k)}=n-\nu_0(H^c)-1$. By assumption $H$ has at least two connected components, and at least one of its component, say $H_1=K_{1,d}$, is not $K_2$, that is $d\geq 2$.  Let $a_2$ be the center and $a_1,b_1$ be leaves of $H_1$. Pick a vertex $b_2$ in a connected component of $H$ different from $H_1$. Then it is easily seen that the edges $\{a_1,b_1\}$ and $\{a_2,b_2\}$ of $H^c$ form an ordered matching of $H^c$.    
 Thus $\nu_0(H^c)\geq 2$. 
 Since any ordered matching of $H^c$ of size $r$ gives an independent set of $H^c$ of size $r$ and the independence  number of $H^c$ is $2$, it follows that $\nu_0(H^c)\leq 2$. This gives the desired result.
\end{proof}

In light of Theorem~\ref{second-power}, it is natural to ask whether the higher
symbolic powers of $I_c(G)$ also have linear quotients. The following example
shows that this is not the case in general, even if $G$ is connected. 

\begin{ex}
Let $G$ be the graph on $[6]$ with $$E(G)=\{\{1,2\},\{1,3\},\{2,3\},\{1,4\},\{4,5\},\{4,6\},\{5,6\}\}.$$ Using Macaulay2, we can verify that $I_{c}(G)^{(3)}$ does not have linear quotients.
\end{ex}
 
It was shown in \cite[Proposition~5.5]{fmm2026} that the symbolic depth function
$k \mapsto \depth S/I_c(G)^{(k)}$ need not be non-increasing and monotonicity fails for a
suitable complete multipartite graph $G$. Corollary~\ref{comparison of 3rd and 4th} shows that
the same phenomenon already occurs and $\depth S/I_c(G)^{(3)}<\depth S/I_c(G)^{(4)}$ for graphs $G$ with $\reg I(G^c)>3$. Nevertheless,
Corollary~\ref{Serre depth Decrease} extends the range of exponents for an arbitrary squarefree monomial ideal,  beyond that covered by
\cite[Theorem~2.7]{nt2019}. We first recall the latter result of
Nguyen and N.V.~Trung.

\begin{thm}[\cite{nt2019}, Theorem 2.7]
Let $I$ be a squarefree monomial ideal. For all $t\geq1$, we have $$\depth S/I^{(t)}\geq\depth S/I^{(k)}\mbox{ for all } k\in E(t),$$
where $E(t)=\bigcup_{i\geq1}\{k\in\mathbb{N}\,\,: i(t-1)+1\leq k\leq it\}$. 
\end{thm}

Recall that the {\em big height} of an ideal $I\subset S$ is defined as
$\bight I=\max\{\height P: P\in \Min(I)\}$.
The following description of the Hilbert series of the module $K_{S/I^{(k)}}^{i}$ will be used in the proof of Theorem~\ref{ineq of dual mod}. 

\begin{cor}[\cite{kmt2025}, Corollary 3.3] \label{dual}
Let $\Delta$ be a simplicial complex and $k\geq1$. Then the Hilbert series of $K^{i}_{S/I_{\Delta}^{(k)}}$ is given by 
\begin{eqnarray*}
F\big(K^{i}_{S/I_{\Delta}^{(k)}},t\big)
& = & \sum_j \dim _{k}\big[K^{i}_{S/I_{\Delta}^{(k)}}\big]_{j}\ t^{j}\\[-0.2cm]
& = & \sum_{F \in \Delta} \big(\sum_{{\bf a}\in \mathbb{N}^{[n]\setminus F}} \dim_k \widetilde{H}_{i-|F|-1}\big(\Delta_{{\bf a}}(I_{{\rm link}_ {\Delta}F}^{(k)});k\big)t^{-|{\bf a}|}\left(\frac{t}{1-t}\right)^{|F|}\big),
\end{eqnarray*}
where $ \mathbb{N}^{V({\rm link}_{\Delta}F)}\hookrightarrow  \mathbb{N}^{V(\Delta)}$ is induced by $ V({\rm link}_{\Delta}F)\hookrightarrow  V(\Delta)$. 
\end{cor}

For a squarefree monomial ideal $I$, using Corollary \ref{dual}, we obtain a stronger result that extends the depth comparison to the Krull dimension of the dual modules of the local cohomology module $H_{\mathfrak m}^{i}(S/I^{(k)})$. 

\begin{thm}\label{ineq of dual mod}
Let $I$ be a squarefree monomial ideal $I$, and let $t\geq1$. Then, we have 
$$\dim K_{S/I^{(t)}}^{i}\leq\dim K_{S/I^{(k)}}^{i}$$
for all $i\geq0$ and $k\geq(\bight I-1)t-(\bight I-2)$.
\end{thm}
\begin{proof}
Set $\Delta=\Delta(I)$ and $p=\bight I$. From Corollary \ref{dual}, for $i$ and $\ell$, we know 
$$\dim K_{S/I^{(\ell)}}^{i}=\max\{|\supp_{-}{\bf a}|\,\,: H_{\mathfrak{m}}^{i}(S/I^{(\ell)})_{\bf a}\neq0\}.$$
Hence, it is enough to prove that for any ${\bf a}\in\mathbb{Z}^{n}$, there exists ${\bf b}\in\mathbb{Z}^{n}$ such that $\supp_{-}{\bf a}=\supp_{-}{\bf b}$ and $\Delta_{\bf a}(I_{\Delta}^{(t)})=\Delta_{\bf b}(I_{\Delta}^{(k)})$ for any $k\geq1$ and either $k=t$ or $k\geq(p-1)t-(p-2)$. It is enough to consider the case that $k\geq(p-1)t-(p-2)$. Fix ${\bf a}=(a_{1},\ldots, a_{n})\in\mathbb{Z}^{n}$. If $\supp_{-}{\bf a}\notin\Delta$, then we see that $\Delta_{\bf a}(I_{\Delta}^{(m)})=\emptyset$ for all $m\geq1$. Hence, we may assume that $\supp_{-}{\bf a}\in\Delta$. Set $\Gamma=\link_{\Delta}\supp_{-}{\bf a}$. From the proof in \cite[Theorem 3.2]{kmt2025}, we have $$\Delta_{\bf a}(I_{\Delta}^{(m)})=\Delta_{\bf a^{+}}(I_{\Gamma}^{(m)})\mbox{ for all }m\geq1.$$
Notice that $\bight I_{\Gamma}\leq p$. Set $V=[n]\setminus\supp_{-}{\bf a}$ and define a vector ${\bf c}\in\mathbb{N}^{n}$ such that $c_{i}=\lfloor\frac{k}{t}a_{i}^{+}+\frac{p-1}{p}\rfloor$ if $i\in V$ and $c_{i}=0$ otherwise. Then, we claim that $\Delta_{\bf a^{+}}(I_{\Gamma}^{(t)})=\Delta_{\bf c}(I_{\Gamma}^{(k)})$. From Lemma \ref{deg cpx}, it is enough to prove that $\sum_{i\in V\setminus F}a_{i}^{+}\leq t-1$ if and only if $\sum_{i\in V\setminus F}c_{i}\leq k-1$ for each facet $F\in\mathcal{F}(\Gamma)$. Fix $F\in\mathcal{F}(\Gamma)$. First suppose that $\sum_{i\in V\setminus F}a_{i}^{+}\leq t-1$. We distinguish the following cases: 

{\bf Case 1.} Suppose that $k\leq(p-1)t$. Since $a_{i}^{+}\geq0$ and $(p-1)a_{i}^{+}\in\mathbb{Z}$ for each $i$, we have 
$$c_{i}=\left\lfloor\frac{k}{t}a_{i}^{+}+\frac{p-1}{p}\right\rfloor\leq\left\lfloor (p-1)a_{i}^{+}+\frac{p-1}{p}\right\rfloor=(p-1)a_{i}^{+},$$
and hence, we see that 
$$\sum_{i\in V\setminus F}c_{i}\leq (p-1)\sum_{i\in V\setminus F}a_{i}^{+}\leq(p-1)(t-1)\leq k-1.$$
{\bf Case 2.} Suppose that $k>(p-1)t$. Then we have 
$$\sum_{i\in V\setminus F}c_{i}\leq\frac{k}{t}\sum_{i\in V\setminus F}a_{i}^{+}+\frac{p-1}{p}|V\setminus F|\leq\frac{k}{t}(t-1)+(p-1)<k,$$
which implies that $\sum_{i\in V\setminus F}c_{i}\leq k-1$. 

Conversely, suppose that $\sum_{i\in V\setminus F}a_{i}^{+}\geq t$. Since $t\geq1$, we know $V\setminus F\neq\emptyset$, which implies that 
$$c_{i}>\frac{k}{t}a_{i}^{+}+\frac{p-1}{p}-1=\frac{k}{t}a_{i}^{+}-\frac{1}{p},$$
and hence, by the fact that $|V\setminus F|\leq p$, we obtain that 
$$\sum_{i\in V\setminus F}c_{i}>\frac{k}{t}\sum_{i\in V\setminus F}a_{i}^{+}-\frac{|V\setminus F|}{p}\geq k-\frac{|V\setminus F|}{p}\geq k-1.$$
Thus, we obtain that $\Delta_{\bf a^{+}}(I_{\Gamma}^{(t)})=\Delta_{\bf c}(I_{\Gamma}^{(k)})$. Define a vector ${\bf b}=(b_{1},\ldots, b_{n})\in\mathbb{Z}^{n}$ as $b_{i}=c_{i}$ if $i\in V$, $b_{j}=a_{j}$ otherwise. Then, we can verify that $\supp_{-}{\bf b}=\supp_{-}{\bf a}$ and $$\Delta_{\bf b}(I_{\Delta}^{(k)})=\Delta_{\bf c}(I_{\Gamma}^{(k)})=\Delta_{\bf a^{+}}(I_{\Gamma}^{(t)})=\Delta_{\bf a}(I_{\Delta}^{(t)}),$$
which completes the proof. 
\end{proof}

Theorem \ref{ineq of dual mod} gives a comparison result for Serre depth and, as a consequence, extends  \cite[Theorem 2.7]{nt2019} on the depth of symbolic powers.

\begin{cor}\label{Serre depth Decrease}
Let $I$ be a squarefree monomial ideal, and let $r\geq2$. Then for any $t\geq 1$ we have 
$$S_{r}\mbox{-}\depth S/I^{(t)}\geq S_{r}\mbox{-}\depth S/I^{(k)}\mbox{ for any } k\geq (\bight I-1)t-(\bight I-2).$$
In particular, we have 
$$\depth S/I^{(t)}\geq\depth S/I^{(k)}\mbox{ for any } k\geq (\bight I-1)t-(\bight I-2).$$
\end{cor}

\smallskip

It follows from Theorem~\ref{min prime} that  a complementary edge ideal $I_c(G)$ satisfies $\bight I_c(G)\leq3$. So in particular, we obtain the following result from Corollary \ref{Serre depth Decrease}. 

\begin{cor}\label{almost non-increasing}
For any graph $G$ and any $t\geq 1$, we have
$$\depth S/I_{c}(G)^{(t)}\geq\depth S/I_{c}(G)^{(k)}\mbox{ for any } k\geq 2t-1.$$
\end{cor}

Moreover, for cover ideals, we recover \cite[Theorem 3.2]{hktt2017}.

\begin{cor}\label{depth Decrease}
Let $G$ be a graph. Then we have 
$$\depth S/J(G)^{(k)}\geq\depth S/J(G)^{(k+1)}\mbox{ for all }k\geq1.$$
\end{cor}
\begin{proof}
Since the cover ideal $J(G)$ of $G$ is unmixed with $\height J(G)=2$, from Corollary \ref{Serre depth Decrease}, we obtain that 
$$\depth S/J(G)^{(t)}\geq\depth S/J(G)^{(k)}\mbox{ for any } k\geq t,$$
which completes the proof. 
\end{proof}

Also, if $I_{c}(G)$ is unmixed, then the symbolic depth function of $I_{c}(G)$ is non-increasing. 

\begin{prop}
Let $G$ be a graph. If $I_{c}(G)$ is unmixed, then we have 
$${\rm depth}(S/I_{c}(G)^{(k)})\geq{\rm depth}(S/I_{c}(G)^{(k+1)})\mbox{ for all }k\geq1.$$
\end{prop}
\begin{proof}
If $\height I_{c}(G)=1$, then $I_{c}(G)$ is a principal ideal and the assertion is clear.  In the case that $\height I_{c}(G)=2$, the ideal $I_{c}(G)$ is the cover ideal of $G^c$. Hence, by \cite[Theorem 3.2]{hktt2017} the result holds. If $\height I_{c}(G)=3$, then we have $G=K_{n}$. One can see that $I_{c}(K_{n})$ is matroidal. Therefore, by \cite[Theorem 2.1]{v2011}, the assertion follows. 
\end{proof}

The following result gives a positive answer to the symbolic-power part of \cite[Question 4.10]{fm2025}. 

\begin{prop}\label{Jc(G)}
Let $G$ be a graph. Then we have 
$$\depth S/(I_{c}(G)^{\vee})^{(k)}\geq\depth S/(I_{c}(G)^{\vee})^{(k+1)}\mbox{ for all }k\geq1.$$
\end{prop}
\begin{proof}
Since $\dim S/I_{c}(G)^{\vee}=2$, the result follows from \cite[Theorem 1.7]{mt2025}. 
\end{proof}

To obtain a more precise description of the behavior of the symbolic depth function of a complementary edge ideal, in the following we compute the depths of its third and fourth symbolic powers. Examples exhibiting a strict drop from the third and fourth symbolic powers were given in \cite[Proposition 5.5]{fmm2026}. We will determine the size of this drop explicitly in terms of the regularity of the edge ideal of complement of the graph. The next lemma is crucial to this aim.

\begin{lemma}\label{deletion-lemma}
Let $G$ be a graph on the vertex set $[n]$ without isolated vertices and  ${\bf a}\in \mathbb{Z}^n$. With the notation introduced before Lemma  \ref{hyperlemma},  let $\Delta=\Delta(\mathcal{H}_{\bf a}^{(k)}(G))$,  $W\subseteq U$ and  $W'=W\setminus \{v:\,  a_v=0\}$. Then the following statements hold.
	\begin{enumerate}
		\item[\rm(a)] Let $v\in W$ be a vertex with $a_{v}=0$. Then $\link_{\Delta_W}v$ is the clique complex of some split graph. In particular, $\widetilde{H}_{i}(\Delta_W;K)\cong\widetilde{H}_{i}(\Delta_{W'};K)$ for all $i\geq2$.
		\item[\rm(b)] Assume $k=3$. If $a_{v}\geq2$ for some $v\in W'$, then  $\Delta_{W'}$ is a cone with
		apex $v$. In particular,  $\widetilde{H}_{i}(\Delta_{W'};K)=0$ for all $i$.
		\item[\rm(c)]  Assume $k=3$ and $\widetilde{H}_{i}(\Delta_W;K)\neq0$ for some $i\geq 2$. Then
		$a_{v}=1$ for every $v\in W'$, in which case $\Delta_{W'}=\Delta(I(G^{c}))_{W'}$.
        \item[\rm(d)] For every $i\geq2$, we have $\widetilde{H}^{i}(\Delta(\mathcal{H}_{\bf a}^{(3)}(G));K)\neq0$ for some ${\bf a}\in\mathbb{Z}^{n}$ if and only if $\widetilde{H}^{i}(\Delta(I(G^{c}))_{A};K)\neq0$ for some $A\subset[n]$.
	\end{enumerate}
\end{lemma}
\begin{proof}
(a) Fix $v\in W$ with $a_{v}=0$. First, we claim that, for any distinct vertices $i,j\in\link_{\Delta_{W}}v$, one has $\{i,j\}\in\link_{\Delta_{W}}v$ if and only if $a_{i}+a_{j}\geq k$. Suppose that $a_{i}+a_{j}\leq k-1$. Then, since $i, j\in U$, we have $0\leq a_{i}, a_{j}\leq k-1$. If $\{i,v\}\notin E(G)$, then from $a_{i}+a_{v}=a_{i}\leq k-1$, we have $\{i,v\}\in E(\mathcal{H}_{\bf a}^{(k)}(G))$, a contradiction. Hence, we obtain $\{i,v\}\in E(G)$ and $\{j, v\}\in E(G)$. If $\{i,j\}\in E(G)$, then  $\{i,j,v\}\in\mathcal{T}(G)$ and $a_{i}+a_{j}+a_{v}\leq k-1$, and hence $\{i,j,v\}\in E(\mathcal{H}_{\bf a}^{(k)}(G))$, which implies that $\{i,j\}\notin\link_{\Delta_{W}}v$. Also if $\{i,j\}\notin E(G)$, then $\{i,j\}\in E(\mathcal{H}_{\bf a}^{(k)}(G))$, which again gives $\{i,j\}\notin\link_{\Delta_{W}}v$.  Conversely, suppose that $a_{i}+a_{j}\geq k$. Then, we know that $\{i,j\}\notin E(\mathcal{H}_{\bf a}^{(k)}(G))$ and $\{i,j,v\}\notin E(\mathcal{H}_{\bf a}^{(k)}(G))$. Hence, $\{i,j,v\}\in\Delta_{W}$, which means that $\{i,j\}\in\link_{\Delta_{W}}v$. 

Next, we claim that the 1-skeleton $(\link_{\Delta_{W}}v)^{(1)}$ is a split graph, that is its vertex set can be partitioned into an independent set and a clique of the graph. Take vertices $i,j\in V(\link_{\Delta_{W}}v)$ such that $i\neq j$, $2a_{i}<k$, and $2a_{j}<k$. 
Then, we have $a_{i}+a_{j}<k$, which implies that $\{i,j\}\notin\link_{\Delta_{W}}v$. Hence, we see that $\{i\in V(\link_{\Delta_{W}}v)\,\,: 2a_{i}<k\}$ is an independent set. On the other hand, take vertices $i,j\in V(\link_{\Delta_{W}}v)$ such that $i\neq j$, $2a_{i}\geq k$, and $2a_{j}\geq k$. Then we have $a_{i}+a_{j}\geq k$, which implies that $\{i,j\}\in\link_{\Delta_{W}}v$. Hence, we see that $\{i\in V(\link_{\Delta_{W}}v)\,\,: 2a_{i}\geq k\}$ is a clique, and so $(\link_{\Delta_{W}}v)^{(1)}$ is a split graph. 

We now prove that $\link_{\Delta_{W}}v=\operatorname{Cl}((\link_{\Delta_{W}}v)^{(1)})$, where $\operatorname{Cl}((\link_{\Delta_{W}}v)^{(1)})$ is the clique complex of $(\link_{\Delta_{W}}v)^{(1)}$. Take a clique $F$ of the graph $(\link_{\Delta_{W}}v)^{(1)}$. Notice that $F\subset V(\link_{\Delta_{W}}v)$. If $F=\{\emptyset\}$, or $F=\{i\}$ for some $i$, then we trivially have $F\in\link_{\Delta_{W}}v$. From now on, we assume that $|F|\geq2$.  Take vertices $i,j\in F\cup\{v\}$. If $v\in\{i,j\}$, then, we know $\{i,j\}\in\Delta_{W}$. Also, if $v\notin\{i,j\}$, then, since $i,j\in F$ and $F$ is a clique, we have $\{i,j\}\in\link_{\Delta_{W}}v$, and hence, we obtain $\{i,j,v\}\in\Delta_{W}$. Therefore, we see that $\{i,j\}\notin E(\mathcal{H}_{\bf a}^{(k)}(G))$. Now, take vertices $i,j,\ell\in F\cup\{v\}$. Since $v\notin F$, we may assume that $i,j\in F$. Now, since $F$ is a clique, we know $\{i,j\}\in\link_{\Delta_{W}}v$, which means that $a_{i}+a_{j}\geq k$. Also, we have $a_{\ell}\geq0$. Hence, we obtain that $a_{i}+a_{j}+a_{\ell}\geq k+0=k>k-1$, and thus, we see that $\{i,j,\ell\}\notin E(\mathcal{H}_{\bf a}^{(k)}(G))$. Therefore, we get $F\in\link_{\Delta_{W}}v$, which proves $\link_{\Delta_{W}}v=\operatorname{Cl}((\link_{\Delta_{W}}v)^{(1)})$. 

Since $(\link_{\Delta_{W}}v)^{(1)}$ is a split graph, it is both chordal and cochordal. Hence, from \cite[Theorem 1]{f1990}, $I_{\link_{\Delta_{W}}v}$ has a 2-linear resolution. From Hochster's formula (see Theorem~\ref{thm:hochster}), we obtain that 
\begin{equation}\label{HiLink}
\widetilde{H}_{i}(\link_{\Delta_{W}}v;K)=0\mbox{ for all }i\geq1.    
\end{equation}
Now, using the equalities $\Delta_{W}=\Delta_{W\setminus\{v\}}\cup\st_{\Delta_{W}}v$ and $\Delta_{W\setminus\{v\}}\cap\st_{\Delta_{W}}v=\link_{\Delta_{W}}v$, we have the Mayer--Vietoris exact sequence 
$$\cdots\rightarrow\widetilde{H}_{i}(\link_{\Delta_{W}}v;K)\rightarrow\widetilde{H}_{i}(\Delta_{W\setminus\{v\}};K)\rightarrow\widetilde{H}_{i}(\Delta_{W};K)\rightarrow\widetilde{H}_{i-1}(\link_{\Delta_{W}}v;K)\rightarrow\cdots$$
for all $i\geq1$. Applying (\ref{HiLink}) to this sequence, we conclude that $\widetilde{H}_{i}(\Delta_{W\setminus\{v\}};K)\cong\widetilde{H}_{i}(\Delta_{W};K)$ for every $i\geq2$. Applying this isomorphism successively to the vertices in $W\setminus W'$ we obtain statement (a).

\smallskip

(b) Assume $k=3$ and let $v\in W^{\prime}$ with $a_{v}\geq2$. By definition, we know $a_{u}\geq1$ for each $u\in W^{\prime}$. Fix $F\in\Delta_{W^{\prime}}$ with $v\notin F$. It suffices to prove that $F\cup\{v\}\in\Delta_{W^{\prime}}$. Suppose that $F\cup\{v\}\notin\Delta_{W^{\prime}}$. Then there exists an edge $e\in E(\mathcal{H}_{\bf a}^{(3)}(G))$ such that $e\subset F\cup\{v\}$. Now, since $F\in\Delta_{W^{\prime}}$, we have $f\not\subset F$ for each $f\in E(\mathcal{H}_{\bf a}^{(3)}(G))$. Hence, we see that $e\subset F\cup\{v\}$ and $e\not\subset F$, and hence, we obtain $v\in e$. We distinguish the following cases: 

{\bf Case 1.} Suppose that $|e|=2$. Then there exists $u\in F$ such that $e=\{u,v\}$. Since $e\in E(\mathcal{H}_{\bf a}^{(3)}(G))$ and $|e|=2$, we have $\{u,v\}\notin E(G)$ and $a_{u}+a_{v}\leq 3-1=2$. Since $a_{v}\geq2$ and $a_{u}\geq1$, this leads to a contradiction. 

{\bf Case 2.} Suppose that $|e|=3$. Then there exist vertices $u,w\in F$ such that $e=\{u,v,w\}$. Since $e\in E(\mathcal{H}_{\bf a}^{(3)}(G))$ and $|e|=3$, we have $\{u,v,w\}\in\mathcal{T}(G)$ and $a_{u}+a_{v}+a_{w}\leq 3-1=2$. Since $a_{v}\geq2$, $a_{u}\geq1$, and $a_{w}\geq1$, this leads to a contradiction.

Therefore, we obtain that $F\cup\{v\}\in\Delta_{W^{\prime}}$, as desired. 
\smallskip

(c) Assume that $k=3$ and $\widetilde{H}_{i}(\Delta_{W};K)\neq0$ for some $i\geq2$. 
From (a), we know $\widetilde{H}_{i}(\Delta_{W^{\prime}};K)\cong\widetilde{H}_{i}(\Delta_{W};K)\neq0$. Fix $v\in W^{\prime}$. From $W^{\prime}\subset U$ and the definition of $W^{\prime}$, we know $a_{v}\geq1$. Also (b) and the assumption that $\widetilde{H}_{i}(\Delta_{W};K)\neq0$ for some $i$, yield that $a_{v}\leq1$. Hence,  $a_{v}=1$ for every $v\in W^{\prime}$. Now, we prove that $\Delta_{W^{\prime}}=\Delta(I(G^{c}))_{W'}=\operatorname{Cl}(G_{W^{\prime}})$. Fix $F\in\Delta_{W^{\prime}}$ and $p,q\in F$. Suppose that $\{p,q\}\notin E(G)$. Then, we know $a_{p}+a_{q}=1+1=2=3-1$, and hence we have $\{p,q\}\in E(\mathcal{H}_{\bf a}^{(3)}(G))$, which is a contradiction to the assumption that $F\in\Delta_{W^{\prime}}$. Hence, $F$ is a clique of $G_{W^{\prime}}$. Conversely, fix $F\in\operatorname{Cl}(G_{W^{\prime}})$. Suppose that $F\notin\Delta_{W^{\prime}}$. Then there exists an edge $e\in E(\mathcal{H}_{\bf a}^{(3)}(G))$ such that $e\subset F$. We distinguish the following cases: 

{\bf Case 1.} Suppose that $|e|=2$. Then there exist $p,q\in F$ such that $e=\{p,q\}$. Since $e\in E(\mathcal{H}_{\bf a}^{(3)}(G))$ and $|e|=2$, we have $\{p,q\}\notin E(G)$ and $a_{p}+a_{q}\leq 3-1=2$. On the other hand, since $F$ is a clique of $G_{W^{\prime}}$, we obtain that $\{p,q\}\in E(G)$, a contradiction. 

{\bf Case 2.} Suppose that $|e|=3$. Then there exist vertices $p,q,r\in F$ such that $e=\{p,q,r\}$. Since $e\in E(\mathcal{H}_{\bf a}^{(3)}(G))$ and $|e|=3$, we have $\{p,q,r\}\in\mathcal{T}(G)$ and $3=a_{p}+a_{q}+a_{r}\leq 3-1=2$, a contradiction. 

Therefore, we obtain $F\in\Delta_{W^{\prime}}$, and hence the assertion follows. 

\smallskip

(d) First, suppose that $\widetilde{H}^{i}(\Delta(\mathcal{H}_{\bf a}^{(3)}(G));K)\neq0$ for some ${\bf a}\in\mathbb{Z}^{n}$ and $i\geq2$. Set $\Delta=\Delta(\mathcal{H}_{\bf a}^{(3)}(G))$ and $U=[n]\setminus\supp_{-}{\bf a}$. Set $A=U\setminus\{v\,\,: a_{v}=0\}$. Then $\widetilde{H}_{i}(\Delta;K)\neq0$, and it follows from (c) that $a_{v}=1$ for every $v\in A$ and $\Delta_{A}=\Delta(I(G^{c}))_{A}$. Combining this with (a), we obtain $\widetilde{H}_{i}(\Delta;K)\cong \widetilde{H}_{i}(\Delta(I(G^{c}))_{A};K)$  which gives $\widetilde{H}_{i}(\Delta(I(G^{c}))_{A};K)\neq 0$. 

Conversely, suppose that $\widetilde{H}_{i}(\Delta(I(G^{c}))_{A};K)\neq 0$ for some $A\subset [n]$ and $i\geq 2$. Let ${\bf a}\in \mathbb{Z}^n$ be a vector defined by $a_i=1$ for any $i\in A$ and $a_i=0$, otherwise, and set $\Delta=\Delta(\mathcal{H}_{\bf a}^{(3)}(G))$.  By (a), we have $\widetilde{H}_{i}(\Delta;K)\cong \widetilde{H}_{i}(\Delta_A;K)$. Also since $a_i=1$ for all non-zero entries of ${\bf a}$, the proof of (c) shows that $\Delta_A=\operatorname{Cl}(G_A)=\Delta(I(G^c))_A$.  Therefore, $\widetilde{H}_{i}(\Delta;K)\cong \widetilde{H}_{i}(\Delta(I(G^c))_A;K)\neq 0$, and hence, we obtain $\widetilde{H}^{i}(\Delta;K)\neq0$, which completes the proof. 
\end{proof}

\begin{thm}\label{3rd symbolic power}
Let $G$ be a graph on the vertex set $[n]$ with $n\geq 3$, and let $A$ be the set of isolated vertices of $G$. Then, we have 
$$\depth S/I_{c}(G)^{(3)}=
\begin{cases}
n-1, & \text{if } G-A\cong K_2,\\[2mm]
n-2, & \parbox[t]{8cm}{if $G-A$ is a star or a disjoint union of edges,\\
and $G-A\not\cong K_2$,}\\[2mm]
n-\max\{3,\operatorname{reg} I(G^c)\}, & \text{otherwise}.
\end{cases}
$$
\end{thm}
\begin{proof}
If $G-A\cong K_2$, then by Corollary~\ref{asymptotic}(a), the assertion holds. So we assume that $G-A\not\cong K_2$. Set $I=I_{c}(G)$. First assume that $G$ has no isolated vertices, that is $A=\emptyset$.
If $G$ is a star, or a disjoint union of edges, since $G\not\cong K_2$, by Corollary~\ref{asymptotic}(a), the equality holds. If $G$ is a complete graph, then $I_{c}(G)$ is the Stanley-Reisner ideal of a matroid. So by~\cite[Theorem 2.1]{v2011}, the ring $S/I_{c}(G)^{(3)}$ is Cohen--Macaulay and so $\depth S/I_{c}(G)^{(3)}=n-\height(I_c(G))=n-3=n-\max\{3,\reg I(G^c)\}$, as $I(G^c)=(0)$. 

Now, assume that $G$ is neither a star graph, nor a disjoint union of edges or a complete graph. Then $I(G^{c})\neq(0)$, and hence, from Hochster's formula, we have
\begin{equation}\label{depthReg}
\reg I(G^{c})=\max\{i+2\,\,: \widetilde{H}^{i}(\Delta(I(G^{c}))_{W});K)\neq0\mbox{ for some }W\subset[n]\}.
\end{equation}
Assume $\reg I(G^{c})\geq4$. We need to show that  $\depth S/I^{(3)}=n-\reg I(G^{c})$. By (\ref{depthReg}), Lemma \ref{deletion-lemma}(d) and  Lemma \ref{hyperlemma}, there exists ${\bf a}\in\mathbb{Z}^{n}$ such that $$H_{\mathfrak{m}}^{n-\reg(I(G^{c}))}(S/I^{(3)})_{\bf a}\cong\widetilde{H}^{\reg I(G^{c})-2}(\Delta(\mathcal{H}_{\bf a}^{(3)}(G));K)\neq0,$$ which gives $\depth S/I^{(3)}\leq n-\reg I(G^{c})$. Let $i<n-\reg I(G^{c})$, and by contradiction assume that $H_{\mathfrak{m}}^{i}(S/I^{(3)})_{\bf a}\neq0$ for some ${\bf a}\in\mathbb{Z}^{n}$. Then Lemma \ref{hyperlemma} and Lemma \ref{deletion-lemma}(d) imply that $\widetilde{H}^{n-i-2}(\Delta(I(G^{c}))_{W});K)\neq0$ for some $W\subset [n]$. Hence, by (\ref{depthReg}) we obtain $\reg I(G^{c})\geq n-i>\reg I(G^{c})$, a contradiction. Thus $H_{\mathfrak{m}}^{i}(S/I^{(3)})_{\bf a}=0$ for all $i<n-\reg I(G^{c})$ and so
 $\depth S/I^{(3)}\geq n-\reg I(G^{c})$. This proves the equality for the case that $\reg I(G^{c})\geq4$.  
Next, suppose that $\reg I(G^{c})\leq3$. From Hochster's formula, we know $\widetilde{H}^{i}(\Delta(I(G^{c}))_{B};K)=0$ for every $B\subset[n]$ and $i\geq2$. Hence, from Lemmas \ref{hyperlemma} and \ref{deletion-lemma}(d), for any $0\leq i<n-3$ and ${\bf a}\in\mathbb{Z}^{n}$, we obtain that 
$$H_{\mathfrak{m}}^{i}(S/I^{(3)})_{\bf a}\cong\widetilde{H}^{n-i-2}(\Delta(\mathcal{H}_{\bf a}^{(3)}(G)); K)=0,$$
which implies that $\depth S/I^{(3)}\geq n-3$.
We show that $\depth S/I^{(3)}\leq n-3$. First assume that $G$ is not a disjoint union of star graphs. Then by 
Theorem~\ref{PowerTwo} and Corollary~\ref{almost non-increasing}, applied to $t=2$ we obtain $\depth S/I^{(3)}\leq \depth S/I^{(2)}=n-3$, and we are done. Now, let $G$ be a disjoint union of star graphs  $K_{1,n_i}$, $1\leq i\leq p$. By our assumption that $G$ is neither a star graph nor disjoint union of edges, we have $p\geq 2$ and $n_i\geq 2$ for some $i$. Then as was shown in the proof of Corollary~\ref{asymptotic}(b), we have $\nu_0(G^c)=2$ and by \cite[Theorem 3.4]{hktt2017}, $\depth S/I^{(k)}=\depth S/J(G^c)^{(k)}=n-\nu_0(G^c)-1$ for any $k\geq 2\nu_0(G^c)-1$. This shows that $\depth S/I^{(3)}=n-3$. Since $G$ is chordal, $\reg I(G^c)=2$, which shows the desired equality.

Now, suppose $A\neq \emptyset$, and let $H=G-A$. By \cite[Corollary 3.2]{KV2026} we have $\reg I(G^c)=\reg I(H^c)$. Since $H\not\cong K_2$, this together with (\ref{depth-isolated}) and the first part of proof gives the desired formula. 
\end{proof}

Next, we determine the depth of the fourth symbolic power of a complementary edge ideal, using the following lemma. 

\begin{lemma}\label{deletion-lemma 2}
Let $G$ be a graph on the vertex set $[n]$ without isolated vertices and  ${\bf a}\in \mathbb{Z}^n$. With the notation introduced before Lemma  \ref{hyperlemma},  let $\Delta=\Delta(\mathcal{H}_{\bf a}^{(4)}(G))$,  $W\subseteq U$ and  $W'=W\setminus \{v:\,  a_v=0\}$. Then $link_{\Delta_{W^{\prime}}}v$ is the clique complex of some split graph for any $v\in W^{\prime}$ with $a_{v}=1$. In particular, for $i\geq2$, we have $\widetilde{H}_{i}(\Delta_{W};K)\cong\widetilde{H}_{i}(\Delta_{W^{\prime}};K)=0$. 
\end{lemma}
\begin{proof}
Set $A=\{v\,\,: a_{v}=1\}$ and $B=W^{\prime}\setminus A$. First, we prove that $\link_{\Delta_{T}}v$ is the clique complex of some split graph for any $T\subset W^{\prime}$ and $v\in T\cap A$. Fix $p,q\in\ V(\link_{\Delta_{T}}v)\cap A$. If $\{v,p,q\}\in\mathcal{T}(G)$, then, since $a_{v}+a_{p}+a_{q}=3$, we get $\{v,p,q\}\in E(\mathcal{H}_{\bf a}^{(4)}(G))$. On the other hand, if $\{v,p,q\}\notin\mathcal{T}(G)$, then, we can take vertices $r,s\in\{v,p,q\}$ such that $\{r,s\}\notin E(G)$, and hence, from $a_{r}+a_{s}=2$, we get $\{r,s\}\in E(\mathcal{H}_{\bf a}^{(4)}(G))$. Therefore, we obtain $\{v,p,q\}\notin\Delta_{T}$, which means $\{p,q\}\notin\link_{\Delta_{T}}v$. Thus $V(\link_{\Delta_{T}}v)\cap A$ is an independent set of $(\link_{\Delta_{T}}v)^{(1)}$. Next, fix $p,q\in V(\link_{\Delta_{T}}v)\cap B$. Then $\{v,p\},\{v,q\}\in\Delta_T$, and hence, $\{v,p\},\{v,q\}\notin E(\mathcal{H}_{\bf a}^{(4)}(G))$. Also, since $a_{p}+a_{q}\geq4$ and $a_{v}+a_{p}+a_{q}\geq5$, we have $\{p,q\},\{v,p,q\}\notin E(\mathcal{H}_{\bf a}^{(4)}(G))$. Hence, we see that $\{v,p,q\}\in\Delta_{T}$, which implies that $\{p,q\}\in\link_{\Delta_{T}}v$, and thus, $V(\link_{\Delta_{T}}v)\cap B$ is a clique. Therefore, $(\link_{\Delta_{T}}v)^{(1)}$ is a split graph. 

Next, we prove that $\link_{\Delta_{T}}v=\operatorname{Cl}((\link_{\Delta_{T}}v)^{(1)})$. Fix a clique $F\in(\link_{\Delta_{T}}v)^{(1)}$. As in the proof of Lemma \ref{deletion-lemma}(a), we know that $C\in\Delta_{W^{\prime}}$ for any $C\in\binom{F\cup\{v\}}{2}$. Now, fix $C\in \binom{F\cup\{v\}}{3}$. Since $v\notin F$, we can take $p,q\in C\cap F$. Write $C\setminus\{p,q\}=\{r\}$. Then, since $\{p,q\}\in\link_{\Delta_{T}}v$ and $r\in W^{\prime}$, we see that $a_{p}+a_{q}+a_{r}\geq 3+1=4$, which implies that $C\notin E(\mathcal{H}_{\bf a}^{(4)}(G))$. Hence, $F\cup\{v\}\in\Delta_{W^{\prime}}$, which means that $F\in\link_{\Delta_{T}}v$. Therefore, we obtain $\link_{\Delta_{T}}v=\operatorname{Cl}((\link_{\Delta_{T}}v)^{(1)})$. 

Applying the same Mayer–Vietoris exact sequence as in the proof of Lemma \ref{deletion-lemma}(a), we get an isomorphism $\widetilde{H}_{i}(\Delta_{W^{\prime}\setminus\{v\}};K)\cong\widetilde{H}_{i}(\Delta_{W^{\prime}};K)$ for each $i\geq2$. Iterating this isomorphism, as the vertices of $A$ are successively removed from $W'$,
we obtain $\widetilde{H}_{i}(\Delta_{W^{\prime}};K)\cong\widetilde{H}_{i}(\Delta_{B};K)$ for each $i\geq2$.
On the other hand, since $a_u\geq 2$ for any $u\in B$, it follows that $B\in \Delta$ and so $\Delta_{B}$ is a simplex. Combining this with Lemma~\ref{deletion-lemma}(a) we get $\widetilde{H}_{i}(\Delta_{W};K)\cong\widetilde{H}_{i}(\Delta_{W^{\prime}};K)\cong\widetilde{H}_{i}(\Delta_{B};K)=0$ for each $i\geq2$, as desired. 
\end{proof}

\begin{thm}\label{4th symbolic power}
Let $G$ be a graph on the vertex set $[n]$ with $n\geq 3$, and let $A$ be the set of isolated vertices of $G$. Then, we have 
$$\depth S/I_{c}(G)^{(4)}=
\begin{cases}
n-1, & \text{if } G-A\cong K_2,\\[2mm]
n-2, & \parbox[t]{8cm}{if $G-A$ is a star or a disjoint union of edges,\\
and $G-A\ncong K_2$,}\\[2mm]
n-3, & \text{otherwise}.
\end{cases}
$$
\end{thm}
\begin{proof}
If $G-A\cong K_2$, then by Corollary~\ref{asymptotic}(a), the assertion holds. So we assume that $G-A\ncong K_2$. Set $I=I_{c}(G)$. First assume that $A=\emptyset$.  If $G$ is a star, or a disjoint union of edges, or a complete graph, then the same argument as in the proof of Theorem \ref{3rd symbolic power} applies.  

Now, assume that $G$ is neither a star graph, nor a disjoint union of edges or a complete graph. Fix ${\bf a}\in\mathbb{Z}^{n}$, and set $\Delta=\Delta(\mathcal{H}_{\bf a}^{(4)}(G))$. From Lemma \ref{deletion-lemma 2}, we know $\widetilde{H}^{i}(\Delta;K)=0$ for every $i\geq2$. Fix $0\leq i<n-3$. If $U\neq\emptyset$, then, from Lemma \ref{hyperlemma}, we obtain $H_{\mathfrak{m}}^{i}(S/I^{(4)})_{\bf a}\cong\widetilde{H}^{n-i-2}(\Delta;K)=0$. If $U=\emptyset$, then we can verify that $H_{\mathfrak{m}}^{i}(S/I^{(4)})_{\bf a}=0$. Therefore, we get $H_{\mathfrak{m}}^{i}(S/I^{(4)})=0$, which implies that $\depth S/I^{(4)}\geq n-3$. The converse inequality $\depth S/I^{(4)}\leq n-3$ follows by applying verbatim the argument used in the proof of Theorem \ref{3rd symbolic power}.

The case $A\neq\emptyset$ follows from the first part of the proof applied to $H=G-A$ together with (\ref{depth-isolated}).
\end{proof}

As a corollary, we completely determine the difference of depths of the third and the fourth symbolic powers of a complementary edge ideal. 

\begin{cor}\label{comparison of 3rd and 4th}
Let $G$ be a graph on the vertex set $[n]$ with $n\geq 3$. Then 
$$\depth S/I_{c}(G)^{(4)}-\depth S/I_{c}(G)^{(3)}=\max\{0,\reg I(G^{c})-3\}.$$
\end{cor}

Combining Theorems \ref{PowerTwo}, \ref{4th symbolic power}, \ref{comparison of 3rd and 4th} and Corollary \ref{almost non-increasing}, we obtain the following inequalities for the symbolic depth functions of $I_{c}(G)$.

\begin{cor}\label{comarison-all}
Let $G$ be a graph on $n\geq3$ vertices. Then we have 
$$\depth S/I_{c}(G)^{(2)}\geq\depth S/I_{c}(G)^{(4)}\geq\depth S/I_{c}(G)^{(3)}\geq\depth S/I_{c}(G)^{(k)}\mbox{ for all }k\geq5.$$
\end{cor}
\begin{proof}
The inequality $\depth S/I_{c}(G)^{(3)}\geq\depth S/I_{c}(G)^{(k)}$, for all $k\geq5$, follows from Corollary \ref{almost non-increasing}. The first two inequalities are obtained by combining Theorem \ref{PowerTwo}, Theorem~\ref{4th symbolic power} and Corollary \ref{comparison of 3rd and 4th}.  
\end{proof}

In the following, we study the behavior of depth of symbolic powers of $I_c(G)$ under the multi-whiskering operation. 
Multi-whisker graphs were introduced by Muta, Pournaki and Terai in \cite{mpt2025} as a generalization of whisker graphs. A whisker graph was introduced in \cite{v1990}. Given a graph $G_0$ on the vertex set  $X_{[h]}=\{x_{1}, \ldots, x_{h}\}$ the \textit{whisker graph associated with} $G_{0}$ is $G_{1}=G_{0}[1,\ldots, 1]$ on the vertex set $X_{[h]}\cup Y_{[h]}$, where $Y_{[h]}=\{y_{1},\ldots, y_{h}\}$ and the edge set $$E(G_{1})=E(G_{0})\cup\{\{x_{1},y_{1}\},\ldots, \{x_{h},y_{h}\}\}.$$ Given positive integers $n_{1},\ldots, n_{h}$, the \textit{multi-whisker graph associated with } $G_{0}$ is $G=G_{0}[n_{1},\ldots, n_{h}]$ on the vertex set 
$$X_{[h]}\cup Y, \mbox{ where } Y=\{y_{1,1},\ldots, y_{1, n_{1}}\}\cup\cdots\cup\{y_{h,1},\ldots, y_{h,n_{h}}\}$$ and the edge set $$E(G)=E(G_{0})\cup\{\{x_1,y_{1,1}\},\ldots, \{x_{1},y_{1,n_{1}}\},\ldots, \{x_{h},y_{h,1}\},\ldots, \{x_{h},y_{h,n_{h}}\}\}.$$

In the proof of Theorem~\ref{multi-whisker}, we make use of the notion of a $d$-Leray simplicial complex, introduced in~\cite{km2006}. Recall that a simplicial complex $\Delta$ is said to be {\em $d$-Leray} over a field $K$, if $\widetilde{H}_i(\Gamma;K)=0$ for all induced subcomplexes $\Gamma\subset \Delta$ and all $i\geq d$. Let $L_K(\Delta)$ denote the minimal $d$ such that $\Delta$ is $d$-Leray over $K$.  

\begin{thm}\label{multi-whisker}
Let $G_{0}$ be a graph on the vertex set $X_{[h]}$ without isolated vertices, where $h\geq3$, let $G=G_{0}[n_{1},\ldots, n_{h}]$ be a multi-whisker graph for positive integers $n_{1},\ldots, n_{h}$, and let $T=K[V(G)]$. 
Then we have 
$$\sum_{i=1}^{h}n_{i}+\depth S/I_{c}(G_{0})^{(k)}-2\leq\depth T/I_{c}(G)^{(k)}\leq \sum_{i=1}^{h}n_{i}+\depth S/I_{c}(G_{0})^{(k)}.$$
\end{thm}
\begin{proof}
Set $r=\sum_{i=1}^{h}n_{i}$. First, we prove that $$r+\depth S/I_{c}(G_{0})^{(k)}-2\leq\depth T/I_{c}(G)^{(k)}.$$ 
Fix ${\bf a}=(\alpha_{1},\ldots, \alpha_{h},\beta_{1,1},\ldots, \beta_{h,n_{h}})\in\mathbb{Z}^{h+r}$. Set $U=\{x_{j}\,\,: \alpha_{j}\geq0\}$ and $V=\{y_{i,s}\,\,: \beta_{i,s}\geq0\}$.
Define a graph $G_{\bf a}$ on the vertex set $U\cup V$ and with the edge set 
\[
\begin{aligned}
E(G_{\bf a})={}&
\{\{x_{i},y_{j,s}\} : i\neq j,\,
\alpha_{i}+\beta_{j,s}\leq k-1\} \\
&\cup
\{\{y_{i,s},y_{j,t}\} :
i\neq j \text{ or } s\neq t,\,
\beta_{i,s}+\beta_{j,t}\leq k-1\}.
\end{aligned}
\]
Then we see that $\Delta(\mathcal{H}_{\bf a}^{(k)}(G))=(\Delta(\mathcal{H}_{\bf a}^{(k)}(G_{0}))*2^{V})\cap\Delta(G_{\bf a})$. 
We claim that 
\begin{equation}\label{H_p=0}
\widetilde{H}_{p}(\Delta(G_{\bf a})_{W};K)=0, \, \, \, \text{ for all $W\subset U\cup V$, and all $p\geq2$.} 
\end{equation}
Fix $p\geq2$ and $W\subset U\cup V$. We use induction on $|Y\cap W|$. If $|Y\cap W|=0$, then we see that $\Delta(G_{\bf a})_{W}=\Delta(G_{\bf a}|_{W})$ is a simplex, which implies that $\widetilde{H}_{p}(\Delta(G_{\bf a})_{W};K)=0$. Suppose that $|Y\cap W|>0$ and take $y_{q, s_{0}}$ such that $\beta_{q,s_{0}}=\min\{\beta_{i,s}\,\,: y_{i,s}\in W\}$. Fix $y_{i,s}, y_{j,t}\in V(G_{\bf a}|_{W}-N[y_{q,s_{0}}])$. 
Then, by assumption, we have $\beta_{q,s_{0}}+\beta_{i,s}>k-1$. Also, by the choice of $y_{q,s_{0}}$, we clearly have $\beta_{j,t}\geq\beta_{q,s_{0}}$, and hence we obtain $\beta_{i,s}+\beta_{j,t}\geq\beta_{i,s}+\beta_{q,s_{0}}>k-1$, which implies that $\{y_{i,s},y_{j,t}\}\notin E(G_{\bf a}|_{W}-N[y_{q,s_{0}}])$. 
Next, let $x_{j}\in V(G_{\bf a}|_{W}-N[y_{q,s_{0}}])$ with $j\neq q$. Since $x_{j}$ is not adjacent to $y_{q,s_{0}}$, we have $\alpha_{j}+\beta_{q,s_{0}}>k-1$. Hence, for all $y_{i,s}\in V(G_{\bf a}|_{W}-N[y_{q,s_{0}}])$, we obtain that $\alpha_{j}+\beta_{i,s}\geq\alpha_{j}+\beta_{q,s_{0}}>k-1$. If $i\neq j$, then $\{x_{j}, y_{i,s}\}\notin E(G_{\bf a}|_{W}-N[y_{q,s_{0}}])$. Also, if $i=j$, then $\{x_{j}, y_{j,s}\}\notin E(G_{\bf a})$, which implies that every vertex $x_{j}$ with $j\neq q$ is isolated in $G_{\bf a}|_{W}-N[y_{q,s_{0}}]$. Hence, if $x_{q}\in V(G_{\bf a}|_{W}-N[y_{q,s_{0}}])$, then we have 
$$E(G_{\bf a}|_{W}-N[y_{q,s_{0}}])=\{\{x_{q},y_{i,s}\}\,\,: y_{i,s}\in V(G_{\bf a}|_{W}-N[y_{q,s_{0}}]), i\neq q, \alpha_{q}+\beta_{i,s}\leq k-1\}.$$Otherwise, the edge set is empty. Therefore, $G_{\bf a}|_{W}- N[y_{q,s_{0}}]$ is the disjoint union of at most one star and some isolated vertices, which implies that $\widetilde{H}_{p}(\Delta(G_{\bf a}|_{W}- N[y_{q,s_{0}}]);K)=0$ for all $p\geq1$.
Now, we have $\Delta(G_{\bf a})_{W}=\Delta(G_{\bf a}|_{W}- y_{q,s_{0}})\cup(\{y_{q,s_{0}}\}*\Delta(G_{\bf a}|_{W}- N[y_{q,s_{0}}]))$ and $\Delta(G_{\bf a}|_{W}- y_{q,s_{0}})\cap(\{y_{q,s_{0}}\}*\Delta(G_{\bf a}|_{W}-N[y_{q,s_{0}}]))=\Delta(G_{\bf a}|_{W}- N[y_{q,s_{0}}])$. By the Mayer--Vietoris exact sequence and the fact that $\widetilde{H}_{r}(\Delta(G_{\bf a}|_{W}- N[y_{q,s_{0}}]);K)=0$ for $r=p-1,p$, we obtain the exact sequence
$$\widetilde{H}_{p}(\Delta(G_{\bf a}|_{W}-y_{q,s_{0}});K)\rightarrow\widetilde{H}_{p}(\Delta(G_{\bf a})_{W};K)\rightarrow\widetilde{H}_{p-1}(\Delta(G_{\bf a}|_{W}-N[y_{q,s_{0}}]);K)=0.$$
By induction hypothesis, we have $\widetilde{H}_{p}(\Delta(G_{\bf a}|_{W}- y_{q,s_{0}});K)=0$. Applying this to the above exact sequence proves the claim and hence (\ref{H_p=0}) holds. Next, we prove that 
\begin{equation}\label{H_p(G_0)}
\text{$\widetilde{H}_{p}(\Delta_{\bf a}(G_{0})_{W};K)=0$,\, \, \,  for all $W\subset U$, and $p\geq h-\depth S/I_{c}(G_{0})^{(k)}-1$.}   
\end{equation}
Define a vector ${\bf b}=(b_{1},\ldots, b_{h})\in\mathbb{Z}^{h}$ by $b_{i}=\alpha_{i}$ if $x_{i}\in W$, and $b_{i}=-1$ otherwise for each $i$. Notice that $\Delta(\mathcal{H}_{\bf a}^{(k)}(G_0))_W=\Delta(\mathcal{H}_{\bf b}^{(k)}(G_0))$. 
Fix $p\geq h-\depth S/I_{c}(G_{0})^{(k)}-1$. By Lemma \ref{hyperlemma}, we see that $$0=H_{\mathfrak{m}_{S}}^{h-p-2}(S/I_{c}(G_{0})^{(k)})_{\bf b}\cong\widetilde{H}_{p}(\Delta(\mathcal{H}_{\bf b}^{(k)}(G_0));K)=\widetilde{H}_{p}(\Delta(\mathcal{H}_{\bf a}^{(k)}(G_0))_W;K).$$
Moreover, if $p\geq h-1$, we clearly have $\widetilde{H}_{p}(\Delta_{\bf a}(G_{0})_{W};K)=0$. This completes the proof of (\ref{H_p(G_0)}). 
 
By (\ref{H_p=0}) and (\ref{H_p(G_0)}) we obtain $L_{K}(\Delta(G_{\bf a}))\leq 2$ and  $$L_{K}(\Delta(\mathcal{H}_{\bf a}^{(k)}(G_{0})*2^{V}))\leq h-\depth S/I_{c}(G_{0})^{(k)}-1,$$ respectively. 
Now, applying \cite[Theorem 1.2]{km2006}, we see that 
$$L_{K}(\Delta(\mathcal{H}_{\bf a}^{(k)}(G)))\leq L_{K}(\Delta(\mathcal{H}_{\bf a}^{(k)}(G_{0})*2^{V}))+L_{K}(\Delta(G_{\bf a}))\leq h-\depth S/I_{c}(G_{0})^{(k)}+1.$$
This implies that $\widetilde{H}_{p}(\Delta(\mathcal{H}_{\bf a}(G));K)=0\mbox{ for all }p\geq h-\depth S/I_{c}(G_{0})^{(k)}+1.$ Hence, using Lemma \ref{hyperlemma}, we conclude that 
$$H_{\mathfrak{m}_{T}}^{i}(T/I_{c}(G)^{(k)})\cong\widetilde{H}_{h+r-i-2}(\Delta(\mathcal{H}_{\bf a}^{(k)}(G));K)=0\mbox{ for all }i<r+\depth S/I_{c}(G_{0})^{(k)}-2.$$ This gives $r+\depth S/I_{c}(G_{0})^{(k)}-2\leq\depth T/I_{c}(G)^{(k)}$. 

Next, we prove that $\depth T/I_{c}(G)^{(k)}\leq r+d$, here we set $d=\depth S/I_{c}(G_{0})^{(k)}$. Notice that $H_{\mathfrak{m}_{S}}^{d}(S/I_{c}(G_{0})^{(k)})\neq0.$ Hence, there exists ${\bf b}=(b_{1},\ldots, b_{h})\in\mathbb{Z}^{h}$ such that $H_{\mathfrak{m}_{S}}^{d}(S/I_{c}(G_{0})^{(k)})_{\bf b}\neq0$. By Lemma \ref{hyperlemma}, we obtain $\widetilde{H}_{h-d-2}(\Delta(\mathcal{H}_{\bf b}^{(k)}(G_{0}));K)\neq0$. Define a vector $${\bf a}=(\alpha_{1},\ldots, \alpha_{h},\beta_{1,1},\ldots, \beta_{h,n_{h}})\in\mathbb{Z}^{h+r}$$ such that $\alpha_{i}=b_{i}$ for each $i\in[h]$ and $\beta_{i,j}=-1$ for each $i\in[h]$ and $j\in[n_{i}]$. In this case, we have $V=\emptyset$, and hence, we see that $\Delta(\mathcal{H}_{\bf a}^{(k)}(G))=\Delta(\mathcal{H}_{\bf b}^{(k)}(G_{0}))$. Again by Lemma \ref{hyperlemma}, we obtain $$H_{\mathfrak{m}_{T}}^{r+d}(T/I_{c}(G)^{(k)})_{\bf a}=H_{\mathfrak{m}_{T}}^{h+r-(h-d-2)-2}(T/I_{c}(G)^{(k)})_{\bf a}\neq0,$$which implies that $\depth T/I_{c}(G)^{(k)}\leq r+d$, as required. 
\end{proof}


\section{Classification of the $S_{2}\mbox{-}\depth$ of $I_{c}(G)^{(k)}$}\label{sec5}

In this section, we first give a characterization of complementary edge ideals satisfying Serre's condition $(S_r)$ for each $r\geq2$. Then we completely determine the Serre depth for $(S_{2})$ of symbolic powers of complementary edge ideals. To this end, we first give a combinatorial description of the Serre depth for $(S_{2})$ of symbolic powers of any squarefree monomial ideal, as a generalization of \cite[Proposition 4.9]{mt2025}. 

The following result answers \cite[Question 2.9]{fm2025}, which asks for a characterization of complementary edge ideals satisfying Serre's condition $(S_r)$ for each $r\geq2$. Recently, Hien and Kumar gave a criterion for $(S_{2})$ in~\cite[Theorem 4.2]{hk2026}. 

\begin{prop}\label{S_r-I_c(G)}
Let $G$ be a graph on the vertex set $[n]$, and let $r\geq 2$ be an integer. Then $S/I_c(G)$ satisfies Serre's condition $(S_r)$ if and only if one of the following holds:
\begin{enumerate}
\item $G$ is a complete graph.
\item $G\cong K_{2}\sqcup(n-2)K_{1}$.
\item $G$ is triangle-free, has no induced cycle of length $4,5,\ldots,r+2$, and has no isolated vertices.
\end{enumerate}
\end{prop}
\begin{proof} 
Set $I=I_{c}(G)$. If $G$ is a complete graph, then by \cite[Theorem 2.2]{hqm2026} or \cite[Theorem 2.8]{fm2025}, $I$ is Cohen-Macaulay. Moreover, if $G\cong K_{2}\sqcup(n-2)K_{1}$, since $I$ is principal, we know that $S/I$ is Cohen--Macaulay. 
Now, suppose that $G$ has isolated vertices. Let $A$ be the set of isolated vertices of $G$ and set $H=G-A$. Since $E(G)\ne\emptyset$ and $G\not\cong K_{2}\sqcup(n-2)K_{1}$, we have $|V(H)|\geq3$. Then $E(H^{c})\cup\mathcal{T}(H)\neq\emptyset$. Fix $C\in E(H^{c})\cup\mathcal{T}(H)$ and $w\in A$. By Theorem \ref{primay decomp}, we know that $P_{C}$ is a minimal prime ideal of $I_{c}(H)$. Set $F=[n]\setminus(C\cup\{w\})$. We claim that $\link_{\Delta(I)}F=\langle C,\{w\}\rangle$. Since $I={\bf x}_{A}I_{c}(H)$, the facets of $\Delta(I)$ are of the form $[n]\setminus\{u\}$ and $[n]\setminus D$, where $u\in A$ and $D\in E(H^{c})\cup\mathcal{T}(H)$. A facet $[n]\setminus\{u\}$ with $u\in A$ contains $F$ only when $u=w$. Next, suppose that $[n]\setminus D$ for some $D\in E(H^{c})\cup\mathcal{T}(H)$ contains $F$. Then, we have $D\subset C$. If $|C|=2$, then we have $D=C$. Suppose $|C|=3$. Then $C$ is a triangle, and hence, no two-element subset of $C$ belongs to $E(H^{c})$. Since $D\in E(H^{c})\cup\mathcal{T}(H)$, $D\subset C$ and $2\leq|D|\leq3$, it follows that $|D|=3$, and thus, $D=C$. Thus, the only facets of $\Delta(I)$ containing $F$ are $F\cup C$ and $F\cup\{w\}$, and hence, we obtain $\link_{\Delta(I)}F=\langle C,\{w\}\rangle$. In particular, $\link_{\Delta(I)}F$ is disconnected and has dimension at least 1. Therefore, $S/I$ does not satisfy Serre's condition $(S_{2})$, a contradiction. Thus, we may assume that $G$ has no isolated vertices. Suppose that $G$ is not complete. Then we see that $\indeg I_{\Delta^{\vee}}=2$, where $\Delta=\Delta(I_c(G))$. Therefore, $S/I_{c}(G)$ satisfies Serre's condition $(S_{r})$ if and only if $I_{\Delta^{\vee}}$ satisfies condition $(N_{2,r})$ by \cite[Corollary 3.7]{y2000}. The minimal generators of $I_{\Delta^{\vee}}$ of degree 3 correspond precisely to the triangles of $G$, which implies that $I_{\Delta^{\vee}}$ satisfies condition $(N_{2,r})$ only if $G$ is triangle-free. In that case, by \cite[Corollary 1.3]{hqm2026} or \cite[Corollary 2.2]{fm2025}, we have $I_{\Delta^{\vee}}=I(G^{c})$. By \cite[Theorem 2.1]{eghp2005}, $I(G^{c})$ satisfies condition $(N_{2,r})$ if and only if $G$ has no induced cycle of length $4,5,\ldots, r+2$, as required. 
\end{proof}

To investigate the Serre depth for $(S_{2})$ of squarefree monomial ideals that are not necessarily unmixed, we provide the following formula as a generalization of \cite[Proposition 4.9]{mt2025}. 

\begin{prop}\label{S_2General}
Let $\Delta$ be a simplicial complex and $k\geq1$. Then we have 
$$S_{2}\mbox{-}\depth S/I_{\Delta}^{(k)}=\min\{\delta(\Delta),\Lambda_{k}(\Delta)+1\},$$
where $\delta(\Delta)=\min\{|F|\,:\, F\in\mathcal{F}(\Delta)\}$ and $$\Lambda_{k}(\Delta)=\min \Bigl\{\, |F| :\,
   \substack{
     \displaystyle F\in\Delta,\text{ and } \Delta_{\bf a}(I_{{\rm link}_{\Delta}F}^{(k)}) \text{ is disconnected for some }{\bf a}\in\mathbb{N}^{[n]\setminus F}
   } \,\Bigr\}$$ and if there is no such a face $F$, then we set $\Lambda_{k}(\Delta)=\infty$. 
\end{prop}
\begin{proof}
From Corollary \ref{dual}, we see that 
\begin{equation}\label{dimKi}
\dim K_{S/I_{\Delta}^{(k)}}^{i}=\max\{|F|\,\,: F\in\Delta,\; {\bf a}\in\mathbb{N}^{[n]\setminus F}\mbox{ with }\widetilde{H}_{i-|F|-1}(\Delta_{\bf a}(I_{\link_{\Delta}F}^{(k)});K)\neq0\}.   
\end{equation}
Take a facet $F\in\mathcal{F}(\Delta)$ with $|F|=\delta(\Delta)$. Then, since $\link_{\Delta}F=\{\emptyset\}$, we have $$\widetilde{H}_{-1}(\Delta_{\bf 0}(I_{\link_{\Delta}F}^{(k)});K)\neq0,$$ which implies that $\dim K_{S/I_{\Delta}^{(k)}}^{\delta(\Delta)}\geq|F|>\delta(\Delta)-1$, and thus we get $S_{2}\mbox{-}\depth S/I_{\Delta}^{(k)}\leq \delta(\Delta)$. Suppose that $\Lambda_{k}(\Delta)<\infty$. Then there exist a face $F$ of $\Delta$ such that $|F|=\Lambda_{k}(\Delta)$. In this case, by (\ref{dimKi}), we see that $\dim K_{S/I_{\Delta}^{(k)}}^{|F|+1}\geq |F|=(|F|+1)-1$, which implies that $$S_{2}\mbox{-}\depth S/I_{\Delta}^{(k)}\leq |F|+1=\Lambda_{k}(\Delta)+1.$$ 
We now prove the converse inequality. By definition, $\dim K_{S/I_{\Delta}^{(k)}}^{S_{2}\mbox{-}\depth S/I_{\Delta}^{(k)}}\geq S_{2}\mbox{-}\depth S/I_{\Delta}^{(k)}-1$. Hence, by (\ref{dimKi}), there exist $F\in\Delta$ and ${\bf a}\in\mathbb{N}^{[n]\setminus F}$ such that $|F|\geq S_{2}\mbox{-}\depth S/I_{\Delta}^{(k)}-1$ and $$\widetilde{H}_{S_{2}\mbox{-}\depth S/I_{\Delta}^{(k)}-|F|-1}(\Delta_{\bf a}(I_{\link_{\Delta}F}^{(k)});K)\neq0.$$Thus $S_{2}\mbox{-}\depth S/I_{\Delta}^{(k)}-|F|-1\in\{-1,0\}$. If $S_{2}\mbox{-}\depth S/I_{\Delta}^{(k)}-|F|-1=-1$, then $\link_{\Delta}F=\{\emptyset\}$ and hence $F\in\mathcal{F}(\Delta)$, which implies that $\delta(\Delta)\leq |F|=S_{2}\mbox{-}\depth S/I_{\Delta}^{(k)}$. Also, if $S_{2}\mbox{-}\depth S/I_{\Delta}^{(k)}-|F|-1=0$, then we obtain $\Lambda_{k}(\Delta)\leq |F|=S_{2}\mbox{-}\depth S/I_{\Delta}^{(k)}-1$, as required. 
\end{proof}

In what follows, we establish several lemmas that will be used to determine the Serre depth for $(S_{2})$ of symbolic powers of complementary edge ideals. 

\begin{lemma}\label{S_2depth lemma}
Let $G$ be a graph on $[n]$ without isolated vertices, let $F\in\Delta(I_{c}(G))$ and let $k\geq1$. If there exists ${\bf a}\in\mathbb{N}^{V(\link_{\Delta(I_{c}(G))}F)}$ such that $\Delta_{\bf a}(I_{\link_{\Delta(I_{c}(G))}F}^{(k)})$ is disconnected,  then we have $n-|F|\leq5$. Moreover, if $G|_{[n]\setminus F}$ is triangle-free, then we have $n-|F|\leq4$. \end{lemma}
\begin{proof}
Set $I=I_{c}(G)$, $\Delta=\Delta(I)$, and $U=[n]\setminus F$. If $|U|\leq 3$, both assertions hold. Hence, we may assume that $|U|\geq4$. For every $u\in U$, we have $|F\cup\{u\}|=n-|U|+1\leq n-3$. Since $I$ is generated in degree $n-2$, it follows that $F\cup\{u\}\in\Delta$, and hence, we see that $V(\link_{\Delta}F)=U$. By assumption, there exists ${\bf a}\in\mathbb{N}^{|U|}$ such that $\Delta_{\bf a}(I_{\link_{\Delta}F}^{(k)})$ is disconnected. From \cite[Theorem 2.1]{fm2025} or \cite[Theorem 1.1]{hqm2026}, we see that $\mathcal{F}(\link_{\Delta}F)=\{U\setminus C\,\,: C\in E((G|_{U})^{c})\cup\mathcal{T}(G|_{U})\}$. Set $\mathcal{H}_{\bf a}=\mathcal{H}_{\bf a}^{(k)}(G|_{U})$. Choose facets $U\setminus C_{1}$ and $U\setminus C_{2}$ of $\Delta_{\bf a}(I_{\link_{\Delta}F}^{(k)})$ belonging to different connected components, where $C_{1},C_{2}\in E(\mathcal{H}_{\bf a})$. Since $U\setminus C_{1}$ and $U\setminus C_{2}$ belong to different connected components, we have $(U\setminus C_{1})\cap(U\setminus C_{2})=\emptyset$. Hence, $U=C_{1}\cup C_{2}$, and thus, $n-|F|=|U|\leq|C_{1}|+|C_{2}|\leq6$. If $G|_{U}$ is triangle-free, then we see that $n-|F|=|U|\leq|C_{1}|+|C_{2}|\leq4$. Suppose that $|U|=6$, that is, $|C_{1}|=|C_{2}|=3$. Write $C_{1}=\{u_{1},u_{2},u_{3}\}$ and $C_{2}=\{v_{1},v_{2},v_{3}\}$. We may assume that $a_{u_{1}}\leq a_{u_{2}}\leq a_{u_{3}}$ and $a_{v_{1}}\leq a_{v_{2}}\leq a_{v_{3}}$. By the definition of $E(\mathcal{H}_{\bf a})$, we have $a_{u_{1}}+a_{u_{2}}+a_{u_{3}}\leq k-1$ and $a_{v_{1}}+a_{v_{2}}+a_{v_{3}}\leq k-1$. Since $a_{u_{1}}+a_{u_{2}}\leq 2a_{u_{3}}$, we have 
$$a_{u_{1}}+a_{u_{2}}\leq\frac{2}{3}(a_{u_{1}}+a_{u_{2}}+a_{u_{3}})\leq\frac{2(k-1)}{3}.$$
Similarly, we see that $$a_{v_{1}}\leq\frac{1}{3}(a_{v_{1}}+a_{v_{2}}+a_{v_{3}})\leq\frac{k-1}{3}.$$
Consequently, we obtain that $a_{u_{1}}+a_{u_{2}}+a_{v_{1}}\leq k-1$. In particular, we have $a_{u_{i}}+a_{v_{1}}\leq k-1$ for $i=1,2$. Fix $i\in\{1,2\}$. Suppose that $\{u_{i},v_{1}\}\notin E(G|_{U})$. Then, $\{u_{i},v_{1}\}\in E(\mathcal{H}_{\bf a})$, and thus, $U\setminus\{u_{i},v_{1}\}\in\Delta_{\bf a}(I_{\link_{\Delta}F}^{(k)})$. Since $\{u_{3},v_{2}\}\subset U\setminus\{u_{i},v_{1}\}$, we have $\{u_{3},v_{2}\}\in\Delta_{\bf a}(I_{\link_{\Delta}F}^{(k)})$. Now, since $u_{3}\in C_{1}$ and $v_{2}\in C_{2}$, the edge $\{u_{3},v_{2}\}$ joins the two connected components containing $U\setminus C_{1}$ and $U\setminus C_{2}$, a contradiction. Hence, we obtain that $\{u_{1},v_{1}\},\{u_{2},v_{1}\}\in E(G|_{U})$. Moreover, since $C_{1}\in E(\mathcal{H}_{\bf a})$ and $|C_{1}|=3$, we know $C_{1}\in\mathcal{T}(G|_{U})$. In particular, $\{u_{1},u_{2}\}\in E(G|_{U})$. Hence, $\{u_{1},u_{2},v_{1}\}\in\mathcal{T}(G|_{U})$. Together with the inequality $a_{u_{1}}+a_{u_{2}}+a_{u_{3}}\leq k-1$ proved above, this implies that $\{u_{1},u_{2},v_{1}\}\in E(\mathcal{H}_{\bf a})$. As a consequence, we obtain that $U\setminus\{u_{1},u_{2},v_{1}\}\in\Delta_{\bf a}(I_{\link_{\Delta}F}^{(k)})$ and thus, $\{u_{3},v_{2}\}\in\Delta_{\bf a}(I_{\link_{\Delta}F}^{(k)})$, which leads us to a contradiction. Therefore, $|U|\neq6$. Since $|U|\leq6$, we conclude that $n-|F|=|U|\leq5$, as required. 
\end{proof}

From Lemma \ref{S_2depth lemma}, we obtain the following statement. 
\begin{lemma}\label{lemma for S_2-depth}
Let $G$ be a graph on the vertex set $[n]$ without isolated vertices and $k\geq1$. Then the following statements hold:
\begin{enumerate}
\item[\textup{(a)}] Suppose that $G$ is triangle-free. Then $S_{2}\mbox{-}\depth S/I_{c}(G)^{(k)}\in\{n-2, n-3\}$. In particular, $S_{2}\mbox{-}\depth S/I_{c}(G)^{(k)}=n-3$ if and only if there exist $F\in\Delta(I_{c}(G))$ and ${\bf a}\in\mathbb{N}^{V(\link_{\Delta(I_{c}(G))}F)}$ such that $|F|=n-4$ and $\Delta_{\bf a}(I_{\link_{\Delta(I_{c}(G))}F}^{(k)})$ is disconnected. 
\item[\textup{(b)}] Suppose that $G$ has triangles. Then $S_{2}\mbox{-}\depth S/I_{c}(G)^{(k)}\in\{n-3,n-4\}$. In particular, $S_{2}\mbox{-}\depth S/I_{c}(G)^{(k)}=n-4$ if and only if there exist $F\in\Delta(I_{c}(G))$ and ${\bf a}\in\mathbb{N}^{V(\link_{\Delta(I_{c}(G))}F)}$ such that $|F|=n-5$ and $\Delta_{\bf a}(I_{\link_{\Delta(I_{c}(G))}F}^{(k)})$ is disconnected. 
\end{enumerate}
\end{lemma}
\begin{proof}
Set $\Delta=\Delta(I_{c}(G))$.

(a) Since $G$ is triangle-free, by Theorem~\ref{min prime} we have $\delta(\Delta)=n-2$. Moreover, Lemma~\ref{S_2depth lemma} implies that $\Lambda_{k}(\Delta)\geq n-4$. Hence, it follows from Proposition~\ref{S_2General}   that $S_{2}\mbox{-}\depth S/I_{c}(G)^{(k)}\in\{n-2, n-3\}$.
Moreover, we have  $S_{2}\mbox{-}\depth S/I_{c}(G)^{(k)}=n-3$ if and only if $\Lambda_{k}(\Delta)=n-4$. This together with Lemma~\ref{S_2depth lemma} yields the second assertion.
\smallskip

(b) We have $\delta(\Delta)=n-3$, and by Lemma~\ref{S_2depth lemma} we obtain $\Lambda_{k}(\Delta)\geq n-5$. Thus Proposition~\ref{S_2General} implies that $S_{2}\mbox{-}\depth S/I_{c}(G)^{(k)}\in\{n-3, n-4\}$. Moreover, we have  $S_{2}\mbox{-}\depth S/I_{c}(G)^{(k)}=n-4$ if and only if $\Lambda_{k}(\Delta)=n-5$. This together with Lemma~\ref{S_2depth lemma} shows the second assertion.
\end{proof}

\begin{lemma}\label{S_2Lemma}
Let $G$ be a graph on the vertex set $[n]$ and let $k\geq1$. Set $\Delta=\langle[n]\setminus C\,\,: C\in E(G^{c})\cup\mathcal{T}(G)\rangle$. Then the following statements hold:
\begin{enumerate}
\item[\textup{(a)}] Suppose that $G$ is triangle-free and $n=4$. Then, there exists ${\bf a}\in\mathbb{N}^{4}$ such that $\Delta_{\bf a}(I_{\Delta}^{(k)})$ is disconnected if and only if one of the following holds:
\begin{enumerate}
\item[\textup{(1)}] $G=C_{4}$ and $k\geq1$. 
\item[\textup{(2)}] $G=P_{4}$ and $k\geq2$. 
\item[\textup{(3)}] $G=P_{3}\sqcup K_{1}$ and $k\geq3$. 
\end{enumerate}
\item[\textup{(b)}] Suppose that $G$ has triangles and $n=5$. Then, there exists ${\bf a}\in\mathbb{N}^{5}$ such that $\Delta_{\bf a}(I_{\Delta}^{(k)})$ is disconnected if and only if $G=K_{5}-e$, and $k=7$ or $k\geq9$, where $e\in E(K_{5})$.  
\end{enumerate}
\end{lemma}
\begin{proof}
Set $\Delta_{\bf a}^{(k)}=\mathcal{H}_{\bf a}^{(k)}(G)$.  

(a) Suppose that there exists ${\bf a}\in\mathbb{N}^{4}$ such that $\Delta_{\bf a}(I_{\Delta}^{(k)})$ is disconnected. Choose two facets $[4]\setminus C_{1}$ and $[4]\setminus C_{2}$ belonging to distinct connected components of $\Delta_{\bf a}(I_{\Delta}^{(k)})$. Now, since $G$ is triangle-free, we have $[4]=C_{1}\cup C_{2}$ and $|C_{1}|=|C_{2}|=2$. After relabeling the vertices, we may assume that $C_{1}=\{1,2\}$ and $C_{2}=\{3,4\}$. Set $\mathcal{H}_{\bf a}^{(k)}=\mathcal{H}_{\bf a}^{(k)}(G)$. Since $C_{1},C_{2}\in E(\mathcal{H}_{\bf a}^{(k)})$, we have  $\{1,2\},\{3,4\}\notin E(G)$, $a_{1}+a_{2}\leq k-1$ and $a_{3}+a_{4}\leq k-1$. We claim that if $i\in C_{1}$, $j\in C_{2}$, and $\{i,j\}\notin E(G)$, then we have $a_{i}+a_{j}\geq k$. Indeed, suppose that $a_{i}+a_{j}\leq k-1$. Then we have $\{i,j\}\in E(\mathcal{H}_{\bf a}^{(k)})$, and hence, $[4]\setminus\{i,j\}$ is a facet of $\Delta_{\bf a}(I_{\Delta}^{(k)})$. This implies that $[4]\setminus C_{1}$ and $[4]\setminus C_{2}$ live in the same connected component of $\Delta_{\bf a}(I_{\Delta}^{(k)})$, which is a contradiction. Since $\{1,2\},\{3,4\}\notin E(G)$, we see that $G$ is a spanning subgraph of $K_{2,2}$. We can distinguish the following cases. 

{\bf Case 1.} Suppose that $G$ is isomorphic to a 4-cycle $C_{4}$. In this case, we set ${\bf a}=(0,0,0,0)$. Then we see that $E(\mathcal{H}_{\bf a}^{(k)})=\{\{1,2\},\{3,4\}\}$, and thus, we obtain $\Delta_{\bf a}(I_{\Delta}^{(k)})=\langle\{1,2\},\{3,4\}\rangle$, which is disconnected. 

{\bf Case 2.} Suppose that $G$ is isomorphic to $P_{4}$. Then there exist $i\in\{1,2\}$ and $j\in\{3,4\}$ such that $\{i,j\}\notin E(G)$. Notice that $a_{i}+a_{j}\geq k$. If $k=1$, then, since $a_{1}+a_{2}\leq 0$ and $a_{3}+a_{4}\leq0$, we obtain ${\bf a}=(0,0,0,0)$, a contradiction. Suppose that $k\geq2$. We may assume that $i=1$ and $j=3$. Set ${\bf a}=(k-1,0,1,0)$. Then $a_{1}+a_{2}=k-1$, $a_{3}+a_{4}=1\leq k-1$, and $a_{1}+a_{3}=k$. On the other hand, we know $a_{1}+a_{3}=k>k-1$. Hence, we see that $\Delta_{\bf a}(I_{\Delta}^{(k)})=\langle\{1,2\},\{3,4\}\rangle$, which is disconnected. 

{\bf Case 3.} Suppose that $G$ is isomorphic to $P_{3}\sqcup K_{1}$. We may assume that $E(G)=\{\{1,3\},\{2,3\}\}$. Since $\{1,4\},\{2,4\}\notin E(G)$, we have $a_{1}+a_{4}\geq k$ and $a_{2}+a_{4}\geq k$. Also, since $a_{3}+a_{4}\leq k-1$, we have $a_{4}\leq k-1$, and hence, $a_{1},a_{2}\geq1$. Thus, we get $2\leq a_{1}+a_{2}\leq k-1$, which implies that $k\geq3$. Conversely, let $k\geq3$ and set ${\bf a}=(1,1,0,k-1)$. Then, $a_{1}+a_{2}=2\leq k-1$, $a_{3}+a_{4}=k-1$, and hence, $a_{1}+a_{4}=a_{2}+a_{4}=k$. Hence, we see that $E(\mathcal{H}_{\bf a}^{(k)})=\{\{1,2\},\{3,4\}\}$, which implies that $\Delta_{\bf a}(I_{\Delta}^{(k)})=\langle\{1,2\},\{3,4\}\rangle$ is disconnected. 

{\bf Case 4.} Suppose that $G$ is isomorphic to $2K_{2}$, $K_{2}\sqcup 2K_{1}$, or $4K_{1}$. After relabeling vertices, we may assume that $\{1,4\},\{2,3\}\notin E(G)$. Now, we have $a_{1}+a_{4}\geq k$ and $a_{2}+a_{3}\geq k$. Hence, we obtain $a_{1}+a_{2}+a_{3}+a_{4}\geq2k$. On the other hand, we know $a_{1}+a_{2}+a_{3}+a_{4}\leq2(k-1)$, which contradicts the above inequality. 

Therefore, there exists ${\bf a}\in\mathbb{N}^{4}$ such that $\Delta_{\bf a}(I_{\Delta}^{(k)})$ is disconnected if and only if one of (1),(2),(3) holds, as required. 

(b) Suppose that there exists ${\bf a}\in\mathbb{N}^{5}$ such that $\Delta_{\bf a}(I_{\Delta}^{(k)})$ is disconnected. Choose two facets $[5]\setminus C_{1}$ and $[5]\setminus C_{2}$ belonging to distinct connected components of $\Delta_{\bf a}(I_{\Delta}^{(k)})$. Notice that $C_{1},C_{2}\in E(\mathcal{H}_{\bf a}^{(k)})$ and $([5]\setminus C_{1})\cap([5]\setminus C_{2})=\emptyset$. Hence, $C_{1}\cup C_{2}=[5]$. By the definition of $\mathcal{H}_{\bf a}^{(k)}$, we know $2\leq|C_{i}|\leq3$ for $i=1,2$. First, we suppose that $|C_{1}|=|C_{2}|=3$. Since $|C_{1}\cap C_{2}|=1$, we can write $C_{1}=\{w,p,q\}$ and $C_{2}=\{w,r,s\}$. We may assume that $a_{p}\leq a_{q}$ and $a_{r}\leq a_{s}$. Since $C_{1},C_{2}\in E(\mathcal{H}_{\bf a}^{(k)})$, we have $a_{w}+a_{p}+a_{q}\leq k-1$ and $a_{w}+a_{r}+a_{s}\leq k-1$, which imply that $$a_{p}\leq\frac{k-1-a_{w}}{2}, \mbox{ and }a_{r}\leq\frac{k-1-a_{w}}{2},$$
and hence, we get $a_{w}+a_{p}+a_{r}\leq k-1$. Suppose that $\{p,r\}\notin E(G)$. Then, since $a_{p}+a_{r}\leq k-1$, we have $\{p,r\}\in E(\mathcal{H}_{\bf a}^{(k)})$, and thus, $[5]\setminus \{p,r\}\in\Delta_{\bf a}(I_{\Delta}^{(k)})$. Also, we see that $([5]\setminus\{p,r\})\cap([5]\setminus C_{1})\neq\emptyset$ and $([5]\setminus\{p,r\})\cap([5]\setminus C_{2})\neq\emptyset$, a contradiction. Hence, $\{p,r\}\in E(G)$ and we see that $\{w,p,r\}$ is a triangle of $G$. Since $a_{w}+a_{p}+a_{r}\leq k-1$, we have $\{w,p,r,\}\in E(\mathcal{H}_{\bf a}^{(k)})$ and we see that $([5]\setminus C_{1})\cap([5]\setminus\{w,p,r\})=\{s\}$ and $([5]\setminus C_{2})\cap([5]\setminus\{w,p,r\})=\{q\}$, a contradiction. Therefore, we may assume that $|C_{1}|=2$ and $|C_{2}|=3$ and, after relabeling vertices, that $C_{1}=\{1,2\}$ and $C_{2}=\{3,4,5\}$. Also, we may assume that $a_{1}\leq a_{2}$ and $a_{3}\leq a_{4}\leq a_{5}$. We now prove that $G=K_{5}-\{1,2\}$. By assumption, we have $a_{1}\leq \frac{k-1}{2}$ and $a_{3}\leq\frac{k-1}{3}$. Also, since $a_{4}\leq a_{5}$, we have $2a_{4}\leq a_{4}+a_{5}\leq k-1-a_{3}\leq k-1$, and hence, we get $a_{4}\leq\frac{k-1}{2}$. Moreover, we see that $a_{1}+a_{3}\leq\frac{k-1}{2}+\frac{k-1}{3}\leq k-1$ and $a_{1}+a_{4}\leq k-1$. If $\{1,3\}\notin E(G)$, then, we have $\{1,3\}\in E(\mathcal{H}_{\bf a}^{(k)})$, which is a contradiction to the choice of $C_{1}$ and $C_{2}$. Hence, $\{1,3\}\in E(G)$. By the same argument, we also have $\{1,4\}\in E(G)$. Moreover, if $a_{1}+a_{3}+a_{4}\leq k-1$, then we have $\{1,3,4\}\in E(\mathcal{H}_{\bf a}^{(k)})$, which is a contradiction to the choice of $C_{1}$ and $C_{2}$. Hence, we obtain $a_{1}+a_{3}+a_{4}\geq k$. On the other hand, since $a_{3}+a_{4}+a_{5}\leq k$, we get $a_{5}\leq a_{1}-1$. Hence, since $a_{3}\leq a_{4}\leq a_{5}$, we obtain that $a_{i}\leq a_{1}-1$ for each $i=3,4,5$. Suppose that $\{1,5\}\notin E(G)$. Since $a_{5}\leq a_{1}-1$ and $a_{1}\leq\lfloor\frac{k-1}{2}\rfloor$, we have $a_{1}+a_{5}\leq2a_{1}-1\leq k-2$. Hence, we obtain $\{1,5\}\in E(\mathcal{H}_{\bf a}^{(k)})$, which means that $[5]\setminus\{1,5\}\in\mathcal{F}(\Delta_{\bf a}(I_{\Delta}^{(k)}))$. Since $\{2,3,4\}\cap([5]\setminus C_{1})\neq\emptyset$ and $\{2,3,4\}\cap([5]\setminus C_{2})\neq\emptyset$, we have a contradiction. Hence, $\{1,5\}\in E(G)$. Take $j\in\{3,4,5\}$. Then we have $a_{2}+a_{j}\leq a_{2}+a_{1}-1<k-1$. If $\{2,j\}\notin E(G)$, then we have $\{2,j\}\in E(\mathcal{H}_{\bf a}^{(k)})$, which leads us to a contradiction. Hence, we obtain that $\{2,3\},\{2,4\},\{2,5\}\in E(G)$, and thus, we obtain that $G=K_{5}-\{1,2\}$. Now, since $a_{1}+a_{2}\leq k-1$ and $a_{1}\leq a_{2}$, we know $a_{1}\leq\lfloor\frac{k-1}{2}\rfloor$. Also, we see that $a_{3}+a_{4}\leq\lfloor\frac{2(k-1)}{3}\rfloor$. Therefore, we obtain 
$k\leq\lfloor\frac{k-1}{2}\rfloor+\lfloor\frac{2(k-1)}{3}\rfloor$. Write $k-1=6q+r$, where $0\leq r\leq5$. Then we have 
\[
\left\lfloor\frac{k-1}{2}\right\rfloor+\left\lfloor\frac{2(k-1)}{3}\right\rfloor-(k-1)=
\begin{cases}
q, & r=0,2,3,4,5, \\
q-1, & r=1. 
\end{cases}
\]
Therefore, we obtain that $k=7$ or $k\geq9$. 

Conversely, we suppose that $G=K_{5}-\{1,2\}$ and either $k=7$ or $k\geq9$. Set $a_{1}=\lfloor\frac{k-1}{2}\rfloor$ and $a_{2}=\lceil\frac{k-1}{2}\rceil$. Write $k-1=3q+r$, where $0\leq r\leq2$. Define 
\[
(a_{3},a_{4},a_{5})=
\begin{cases}
(q,q,q), & r=0, \\
(q,q,q+1), & r=1, \\
(q,q+1,q+1), & r=2. 
\end{cases}
\]
Notice that $a_{1}+a_{2}=k-1$, $a_{3}+a_{4}+a_{5}=k-1$, and $a_{3}+a_{4}=\lfloor\frac{2(k-1)}{3}\rfloor$. Since $k=7$ or $k\geq9$, we have 
$a_{1}+a_{3}+a_{4}>k-1$. Also, since $a_{2}\geq a_{1}$ and $a_{3}\leq a_{4}\leq a_{5}$, we see that $\{1,2\}\in E(\mathcal{H}_{\bf a}^{(k)})$ and $\{3,4,5\}\in E(\mathcal{H}_{\bf a}^{(k)})$. Since $G=K_{5}-\{1,2\}$, the only non-edge of $G$ is $\{1,2\}$. Since $a_{1}+a_{2}=k-1$, we have $\{1,2\}\in E(\mathcal{H}_{\bf a}^{(k)})$. Also, since $a_{3}+a_{4}+a_{5}=k-1$, we obtain $\{3,4,5\}\in E(\mathcal{H}_{\bf a}^{(k)})$. Every other triangle of $G$ is of the form $\{p,i,j\}$, where $p\in\{1,2\}$ and $i,j\in\{3,4,5\}$. Since $a_{2}\geq a_{1}$ and $a_{3}\leq a_{4}\leq a_{5}$, we obtain that $a_{p}+a_{i}+a_{j}\geq a_{1}+a_{3}+a_{4}\geq k$. Hence, none of these triangles belongs to $E(\mathcal{H}_{\bf a}^{(k)})$. Therefore, $E(\mathcal{H}_{\bf a}^{(k)})=\{\{1,2\},\{3,4,5\}\}$, that is, $\Delta_{\bf a}(I_{\Delta}^{(k)})=\langle\{1,2\},\{3,4,5\}\rangle$, which is disconnected. Therefore, the assertion follows. 
\end{proof}

The preceding lemmas allow us to establish the main result of this section, which gives a complete classification of $S_{2}\mbox{-}\depth$ of symbolic powers of complementary edge ideals.  Here, $H\subseteq G$ means that $G$ contains an induced subgraph isomorphic to $H$. 

\begin{thm}\label{S_2main}
Let $G$ be a graph on $[n]$ with $n\geq3$, without isolated vertices and $I=I_c(G)$. The following statements hold.
\begin{enumerate}
\item[\textup{(a)}] If $G$ is triangle-free, then
\[
(S_{2}\mbox{-}\depth S/I^{(k)})_{k\geq1}=
\begin{cases}
(n-3,n-3,\ldots), & \text{if } C_4\subseteq G, \\
(n-2,n-3,n-3,\ldots), & \text{if } P_4\subseteq G, \, C_4\nsubseteq G, \\
(n-2,n-2,n-3,n-3,\ldots), & \text{if } P_{3}\sqcup K_{1}\subseteq G,\, C_4\nsubseteq G,\, P_4\nsubseteq G, \\
(n-2,n-2,\ldots), & \text{ otherwise}.
\end{cases}
\]
\item[\textup{(b)}] If $G$ contains a triangle, then
\[
S_{2}\mbox{-}\depth S/I^{(k)}=
\begin{cases}
n-4, &
\text{if } K_{5}-e\subseteq G, 
\text{ and } k=7\mbox{ or }k\geq9,\\ 
n-3, &
\text{otherwise}.
\end{cases}
\]
\end{enumerate}
In particular, the following conditions are equivalent: 
\begin{enumerate}
\item[\textup{(1)}] $K_{5}-e\nsubseteq G$. 
\item[\textup{(2)}] $S_{2}\mbox{-}\depth\, S/I_{c}(G)^{(k)}\geq S_{2}\mbox{-}\depth\, S/I_{c}(G)^{(k+1)}$ for all $k\geq1$.
\end{enumerate}
Moreover, if $G$ contains $K_{5}-e$, then 
$S_{2}\mbox{-}\depth S/I_{c}(G)^{(7)}<S_{2}\mbox{-}\depth S/I_{c}(G)^{(8)}$. 
\end{thm}
\begin{proof}
(a) By Lemma \ref{lemma for S_2-depth}(a), we know that $S_{2}\mbox{-}\depth S/I_{c}(G)^{(k)}=n-3$ if and only if there exist $F\in\Delta(I_{c}(G))$ and ${\bf a}\in\mathbb{N}^{V(\link_{\Delta(I_{c}(G))}F)}$ such that $|F|=n-4$ and $\Delta_{\bf a}(I_{\link_{\Delta(I_{c}(G))}F}^{(k)})$ is disconnected. Set $W=[n]\setminus F$, and define $$\Gamma_{W}=\langle W\setminus C\,\,:C\in E((G|_{W})^{c})\cup\mathcal{T}(G|_{W})\rangle.$$
Then, we see that $\link_{\Delta}F=\Gamma_{W}$. Since $|F|=n-4$, we have $|W|=4$. Moreover, for each $w\in W$, we have $|F\cup\{w\}|=n-3$, which implies that $F\cup\{w\}\in\Delta$, and hence, we see that $V(\link_{\Delta})=W$. Therefore, we obtain that $S_{2}\mbox{-}\depth S/I_{c}(G)^{(k)}=n-3$ if and only if there exist a subset $W\subset[n]$ with $|W|=4$ and a vector ${\bf a}\in\mathbb{N}^{4}$ such that $\Delta_{\bf a}(I_{\Gamma_{W}}^{(k)})$ is disconnected. Thus, from Lemma \ref{S_2Lemma}(a), we have $S_{2}\mbox{-}\depth S/I_{c}(G)^{(k)}=n-3$ if and only if one of the following holds:
\begin{enumerate}[label=\textbullet]
\item  $G$ contains an induced $C_{4}$.
\item  $G$ contains an induced $P_{4}$ and $k\geq2$.
\item  $G$ contains an induced $P_{3}\sqcup K_{1}$ and $k\geq3$.
\end{enumerate}
If $G$ contains an induced $C_{4}$, then we have 
$$(S_{2}\mbox{-}\depth S/I_{c}(G)^{(k)})_{k\geq1}=(n-3,n-3,\ldots).$$ 
Also, if $P_4\subseteq G$ and $C_4\nsubseteq G$, then the above discussion together with Lemma~\ref{lemma for S_2-depth} yields 
$$(S_{2}\mbox{-}\depth S/I_{c}(G)^{(k)})_{k\geq1}=(n-2,n-3,n-3,\ldots).$$ 
Moreover,if $P_{3}\sqcup K_{1}\subseteq G$, $C_4\nsubseteq G$, and $P_4\nsubseteq G$, then we have 
$$(S_{2}\mbox{-}\depth S/I_{c}(G)^{(k)})_{k\geq1}=(n-2,n-2,n-3,n-3,\ldots).$$ 
Finally, if $G$ contains none of $C_{4}, P_{4}$, and $P_{3}\sqcup K_{1}$ as an induced subgraph, then, from Lemma \ref{S_2Lemma} (a), we obtain $\Delta_{\bf a}(I_{\Gamma_{W}}^{(k)})$ is connected for every $W\subset[n]$ with $|W|=4$ and every ${\bf a}\in\mathbb{N}^{|W|}$. Hence, from Lemma \ref{lemma for S_2-depth} (a), we see that 
$$S_{2}\mbox{-}\depth S/I_{c}(G)^{(k)}=n-2\mbox{ for all }k\geq1,$$
which completes the proof. 

(b) By Lemma \ref{lemma for S_2-depth}(b), we know that $S_{2}\mbox{-}\depth S/I_{c}(G)^{(k)}=n-4$ if and only if there exist $F\in\Delta(I_{c}(G))$ and ${\bf a}\in\mathbb{N}^{V(\link_{\Delta(I_{c}(G))}F)}$ such that $|F|=n-5$ and $\Delta_{\bf a}(I_{\link_{\Delta(I_{c}(G))}F}^{(k)})$ is disconnected. Set $W=[n]\setminus F$, and define $$\Gamma_{W}=\langle W\setminus C\,\,:C\in E((G|_{W})^{c})\cup\mathcal{T}(G|_{W})\rangle.$$
Then, we see that $\link_{\Delta}F=\Gamma_{W}$. Since $|F|=n-5$, we have $|W|=5$. Also, we have $V(\link_{\Delta})=W$. Therefore, we obtain that $S_{2}\mbox{-}\depth S/I_{c}(G)^{(k)}=n-4$ if and only if there exist a subset $W\subset[n]$ with $|W|=5$ and a vector ${\bf a}\in\mathbb{N}^{5}$ such that $\Delta_{\bf a}(I_{\Gamma_{W}}^{(k)})$ is disconnected. Moreover, we see that $G|_{W}$ contains, a triangle. Indeed, if $G|_{W}$ is triangle-free, then, from Lemma \ref{S_2depth lemma}, we have $|W|=n-|F|\leq4$, a contradiction. 
Thus, by Lemma \ref{S_2Lemma}(b), $S_{2}\mbox{-}\depth S/I_{c}(G)^{(k)}=n-4$ if and only if there exists a subset $W\subset[n]$ with $|W|=5$ such that $G|_{W}\cong K_{5}-e$ and either $k=7$ or $k\geq9$, where $e\in E(K_{5})$. This shows the desired result and completes the proof. 
\end{proof}

As an immediate consequence, we determine  when the $S_{2}$-depth sequence stabilizes. 

\begin{cor}
Let $G$ be a graph without isolated vertices. Then we have 
$$S_{2}\mbox{-}\depth S/I_{c}(G)^{(k)}=S_{2}\mbox{-}\depth S/I_{c}(G)^{(9)}\mbox{ for all }k\geq9.$$
Moreover, this bound is sharp.
\end{cor}

The following result provides a characterization of the Serre's condition $(S_{2})$ and of the Cohen--Macaulayness of $S/I_{c}(G)^{(k)}$, thereby extending and strengthening \cite[Theorem 4.5]{hk2026} to all symbolic powers of complementary edge ideals, including $k=2$.

\begin{thm}\label{char of CM and S_2}
Let $G$ be a graph on  $n\geq3$ vertices, and let $k\geq2$ be an integer. Then $S/I_{c}(G)^{(k)}$ satisfies Serre's condition $(S_{2})$ if and only if $S/I_{c}(G)^{(k)}$ is Cohen--Macaulay. Moreover, these conditions are equivalent to the following conditions:
\begin{enumerate}
\item[\textup{(1)}] If $k=2$, then either  $G\cong K_{n}$, or $G$ has no isolated vertices and every connected component of $G$ is a star graph, or $G\cong K_{2}\sqcup(n-2)K_{1}$. 
\item[\textup{(2)}] If $k\geq3$, then either $G\cong K_{n}$, or $G$ is a star graph, or $G$ is the disjoint union of edges, or $G\cong K_{2}\sqcup(n-2)K_{1}$. 
\end{enumerate}
In particular, the Cohen--Macaulayness of $S/I_{c}(G)^{(k)}$ does not depend on the characteristic of the base field for all $k\geq1$. 
\end{thm}
\begin{proof}
Set $I=I_{c}(G)$ and fix $k\geq2$. 
Suppose that $S/I^{(k)}$ satisfies Serre's condition $(S_{2})$. We distinguish the following cases. 

{\bf Case 1.} Suppose that $G$ has no isolated vertices. First, if $G\cong K_{n}$, then, since $I$ is the Stanley--Reisner ideal of a matroid complex, we know $S/I^{(k)}$ is Cohen--Macaulay from \cite[Theorem 2.1]{v2011}. Hence, we assume that $G\not\cong K_{n}$. Then, from Theorem \ref{primay decomp}, we know $\dim S/I^{(k)}=n-2$. Hence, from \cite[Theorem A]{fSerre2026}, we see that $I^{(k)}$ satisfies Serre's condition $(S_{2})$ if and only if $S_{2}\mbox{-}\depth S/I^{(k)}=n-2$. If $G$ contains triangles, from Theorem \ref{S_2main} (b), we have $S_{2}\mbox{-}\depth S/I^{(k)}\in\{n-3,n-4\}$. Hence, $G$ has no triangles. Suppose that $k=2$. From Theorem \ref{S_2main}(a) and the assumption that $G$ is triangle-free, we see that $S_{2}\mbox{-}\depth S/I^{(k)}=n-2$ if and only if every connected component of $G$ is a star graph. Next, suppose that $k\geq3$. From Theorem \ref{S_2main}(a), we know $S_{2}\mbox{-}\depth S/I^{(k)}=n-2$ if and only if $G$ does not contain $C_{4}, P_{4}$, and $P_{3}\sqcup K_{1}$ as an induced subgraph. Hence, we see that $S/I^{(k)}$ satisfies Serre's condition $(S_{2})$ if and only if $G$ is a star graph, or a disjoint union of edges. Hence, from Theorem \ref{PowerTwo} and Corollary~\ref{asymptotic}(a), we obtain that $S/I^{(k)}$ is Cohen--Macaulay. 

{\bf Case 2.} Suppose that $G$ has isolated vertices. We prove that $G\cong K_{2}\sqcup(n-2)K_{1}$. Let $A$ be the set of isolated vertices of $G$, and set $H=G-A$. If $H\cong K_{2}$, then, we have $G\cong K_{2}\sqcup(n-2)K_{1}$, and thus, we see that $I_{c}(G)$ is principal, which implies that $S/I_{c}(G)^{(k)}$ is Cohen--Macaulay. If $H\ncong K_{2}$,  then $|V(H)|\geq3$, and we show that this case does not happen. Suppose on the contrary that $|V(H)|\geq3$. Then $E(H^{c})\cup\mathcal{T}(H)\neq\emptyset$. Fix $C\in E(H^{c})\cup\mathcal{T}(H)$ and $w\in A$. By Theorem \ref{primay decomp}, we know that $P_{C}$ is a minimal prime ideal of $I_{c}(H)$.
Set $F=[n]\setminus(C\cup\{w\})$. As in the proof of Proposition \ref{S_r-I_c(G)}, we know that $\link_{\Delta(I)}F=\langle C, \{w\}\rangle$ is disconnected and has dimension at least 1. Since $S/I^{(k)}$ satisfies Serre's condition $(S_{2})$, from \cite[Corollary 3.5 (i)]{tt2012}, we know that $K[V(\link_{\Delta(I)}F)]/I_{\link_{\Delta(I)}F}^{(\ell)}$ also satisfies Serre's condition $(S_{2})$ for any $1\leq\ell\leq k$. In particular, since $k\geq2$, we obtain that $K[V(\link_{\Delta(I)}F)]/I_{\link_{\Delta(I)}F}^{(2)}$ satisfies Serre's condition $(S_{2})$. Moreover, $\dim K[V(\link_{\Delta(I)}F)]/I_{\link_{\Delta(I)}F}=|C|\geq2$, and thus, $\depth K[V(\link_{\Delta(I)}F)]/I_{\link_{\Delta(I)}F}^{(2)}\geq2$. Hence, from \cite[Theorem 3.2]{rty2011}, we obtain the diameter of the 1-skeleton of $\link_{\Delta(I)}F$ at most 2, which is a contradiction to the fact that $w$ is an isolated vertex of $(\link_{\Delta}F)^{(1)}$. Therefore, we obtain $H\cong K_{2}$, and hence, $G\cong K_{2}\sqcup(n-2)K_{1}$. 

Conversely, suppose that $G$ satisfies one of the conditions in (1) when $k=2$, or one of those listed in (2) when $k\geq3$. We show that $S/I^{(k)}$ is Cohen--Macaulay. If $G\cong K_{n}$, then $I$ is the Stanley--Reisner ideal of a matroid complex, and hence, from \cite[Theorem 2.1]{v2011}, we know $S/I^{(k)}$ is Cohen--Macaulay for all $k\geq1$. If $k=2$, $G$ has no isolated vertices, and every connected component of $G$ is a star graph, then from Theorem \ref{PowerTwo}, we know $\depth S/I^{(2)}=n-2=\dim S/I^{(2)}$, which means that $S/I^{(2)}$ is Cohen--Macaulay. If $k\geq3$ and $G$ is a star graph or the disjoint union of edges, then from Corollary \ref{asymptotic}(a), we know $\depth S/I^{(k)}=n-2=\dim S/I^{(k)}$, which means that $S/I^{(k)}$ is Cohen--Macaulay. Finally, if $G\cong K_{2}\sqcup(n-2)K_{1}$, then $I$ is principal and so $S/I^{(k)}$ is Cohen--Macaulay. 

The last statement follows from the characterizations given in (1) and (2), together with \cite[Theorem 2.8]{fm2025}.
\end{proof}


\section*{Acknowledgement}
The research of A. Ficarra was supported by the Grant JDC2023-051705-
I funded by MICIU/AEI/10.13039/501100011033 and by the FSE+ and also by
INDAM (Istituto Nazionale Di Alta Matematica). The research of S. Moradi is funded by Iran National Science Foundation (INSF) under project No. 40507515.
The research of Y. Muta was partially supported by ohmoto-ikueikai and JST SPRING Japan Grant Number JPMJSP2126. 



\end{document}